\documentclass[preprint,12pt]{elsarticle}
\usepackage{amsmath}
\usepackage{amssymb}
\usepackage{amsfonts}
\usepackage{amsthm}
\usepackage{comment}
\usepackage{pgfplots}
\usepackage{graphicx}
\usepackage{amsmath}
\usepackage[T1]{fontenc}
\usepackage{mathrsfs}
\usepackage{latexsym}
\usepackage{amssymb}
\usepackage{graphicx}
\usepackage{float}
\usepackage{extarrows}
\usepackage{epstopdf}
\usepackage{caption}
\usepackage{subcaption}
\usepackage{wrapfig}
\usepackage[colorlinks]{hyperref}
\usepackage{lipsum}
\usepackage{subcaption}
\usepackage{array}

\newtheorem{remark}{Remark}
\newtheorem{theorem}{Theorem}
\newtheorem{definition}{Definition}

\begin{document}

\begin{frontmatter}



\title{High-order structure-preserving methods for damped Hamiltonian system}

\author[inst1]{Lu Li}

\affiliation[inst1]{organization={School of Mathematics (Zhuhai)},
            addressline={Sun Yat-sen University}, 
            city={Zhuhai},
           postcode={519082}, 
            state={Guangdong},
            country={China}}


\begin{abstract}

We present a novel methodology for constructing arbitrarily high-order structure-preserving methods tailored for damped Hamiltonian systems.  This method combines the idea of exponential integrator and energy-preserving collocation methods, effectively preserving the energy dissipation ratio introduced by the damping terms.  We demonstrate the conservative properties of these methods and confirm their order of accuracy through numerical experiments involving the damped Burger’s equation and Korteweg-de-Vries equation.
\end{abstract}



\begin{keyword}
  structure-preserving \sep damped Hamiltonian \sep exponential integrator \sep energy dissipation
\end{keyword}

\end{frontmatter}



\section{Introduction}
In this paper, we consider damped Hamiltonian systems formulated as follows:
\begin{equation}\label{H-damp-ODE}
\dot{x}=S(x)\nabla H(x)-D(t)x,
\end{equation}
where $x\in\mathbb{R}^N$, the matrix $S(x)\in\mathbb{R}^{N\times N}$ is skew-symmetric, the  matrix $D(t)\in\mathbb{R}^{N\times N}$ is  diagonal, and the energy function $H(x)$ maps any vector $x$ to a real number. Such systems exhibit an interplay between conservative and dissipative forces. While conservative forces govern the system's energy conservation and produce regular oscillatory behavior, damping terms introduce energy dissipation and may lead to convergence to equilibrium or other complex behaviors. These systems have become more popular in applications such as  biomechanics \cite{chang2021controlling}  and physics \cite{temam2012infinite}. 

It is well established that the exact solution of an Hamiltonian system without the damping term $D(t)x$, represented by: 
\begin{equation}\label{H-ODE}
\dot{x}=S(x)\nabla H(x),
\end{equation}
preserves the energy, i.e., $\frac{dH(x)}{dt}\mid_{x=x(t)}=0$, and also symplectic structure if $S(x)$ is the canonical skew-symmetric matrix. Extensive research has focused on developing energy-preserving or symplectic numerical methods for such systems, owing to their superior numerical properties  \cite{hairer2006geometric,leimkuhler2004simulating}. When damping terms are included, the modified damped Hamiltonian system  \eqref{H-damp-ODE}  preserves modified structures such as energy dissipation and  conformal symplecticity, provided  $S(x)$ remains a canonical skew-symmetric matrix and and the diagonal matrix  $D(t)$ has constant equal diagonal elements \cite{mclachlan2001conformal}. Numerical methods preserving these structures include conformal symplecticity \cite{mclachlan2002splitting, modin2011geometric, bhatt2016second, sun2005structure,bhatt2017structure} and conformal multi-symplecticity \cite{moore2009conformal,moore2013conformal,bhatt2019exponential,cai2017modelling}, as well as recently developed methods that focus on preserving correct energy dissipation rates \cite{bhatt2021projected,moore2021exponential, uzunca2023linearly}. Conformal symplectic methods  preserve the decay rates in linear and quadratic invariants under certain constraints on the coefficient functions $D(t)$ \cite{bhatt2017structure}, and in general, it is not possible to design numerical methods that can preserve conformal symplecticity and energy dissipation rate simultaneously. The choice of conformal symplectic methods or methods that preserve exact energy property depends on the objective of the study as well as the differential equation.  In this paper, we consider methods that can preserve the energy dissipation rate for damped Hamiltonian systems with more general diagonal matrix $D(t)$.

Suppose $I(x)$ is an invariant of  \eqref{H-ODE}, i.e., $I(x)$ satisfies 
\begin{equation*}
\frac{d}{dt}I(x)=\nabla I(x)^TS(x)\nabla H(x)=0.
\end{equation*}
We consider a special invariant $I(x)$ which satisfies $\nabla I(x)^TD(t)x=\eta(t)I(x)$, where $\eta(t)>0$ is a continuous function. Then for such $I(x)$ and the damped Hamiltonian system \eqref{H-damp-ODE}, we have 
\begin{equation}\label{conformal condition}
\frac{d}{dt}I(x)=-\eta(t)I(x).
\end{equation}
The solution of equation \eqref{H-damp-ODE} satisfies $I(x(t))=e^{-\phi(t)}I(x(0))$ with $\phi(t)=\int_0^{t}\eta(s)ds$, which indicates that $I(x(t))$ dissipates exponentially with the rate $e^{-\int_0^{t}\eta(s)ds}$. This energy property of the exact solution for the damped Hamiltonian system \eqref{H-damp-ODE} motivates the study on numerical methods that can preserve such exponential decay property.

Discrete gradient method has been a popular technique to construct methods that can preserve the conservative or dissipative properties of energy for differential equations \cite{celledoni2012preserving,hairer2014energy,li2016exponential,wang2019exponential,li2022new}. Recent studies have shown that a modification of  the discrete gradient method for Hamiltonian systems \eqref{H-ODE} can exactly preserve the  dissipation rate of a damped Hamiltonian systems when  the energy function $H(x)$ has special forms \cite{moore2021exponential}. Based on this work, it has been shown that the linearly implicit methods,  such as Kahan's method \cite{celledoni2012geometric,eidnes2021linearly} and the method based on a polarization of  the polynomial  Hamiltonian functions \cite{dahlby2011general,eidnes2021linearly}, can be used to construct linearly implicit methods that can preserve the energy dissipation property for a certain type of damped Hamiltonian system \cite{uzunca2023linearly}. However, all the above mentioned  energy-dissipating methods tailed to damped Hamiltonian systems are second-order. We aim to construct high-order  methods that can preserve the dissipation rate of energy for damped Hamiltonian systems. High-order methods are essential for solving differential equations due to their efficiency, accuracy, and stability. They achieve faster convergence, requiring fewer iterations to reach precise solutions, which is crucial for complex and large-scale problems. Despite their complexity, they often reduce overall computational costs. Furthermore, high-order methods can serve as reliable benchmarks for evaluating new numerical techniques, especially when exact solutions are unavailable.

A general technique to construct high-order methods is the splitting technique; however, it may destroy the structure preservation property and the determination of the appropriate order conditions is not always straightforward. Another method is to use the order condition of Runge-Kutta method, which has been shown effective in constructing  methods preserving
symplecticity \cite{hairer2006geometric}, multisymplecticity\cite{mclachlan2014high} or conformal symplecticity \cite{bhatt2017structure,mclachlan2014high}. For constructing high-order methods that can preserve energy, Hairer proposed the energy-preserving collocation method which also uses the order condition of Runge-Kutta method \cite{hairer2010energy}. This method generalizes the averaged vector field method, known as a discrete gradient method, from second to higher order. We draw inspiration from the energy-preserving collocation method.  By incorporating an exponential transformation to the energy-preserving collocation method, we develop new arbitrarily high-order symmetric exponential integrators that can preserve the energy dissipation rate.

The rest of the paper is organised as follows. Section \ref{EPC} reviews energy-preserving collocation methods. Followed is the presentation of the new methods in section \ref{Higher-order EPC-damped}. In  section \ref{Numerical-test}, we test the preservation of energy dissipation rate and confirm the order of the methods on two differential equations. Finally, we conclude the study in the last section.

\section{Energy-preserving Collocation methods for Hamiltonian systems}\label{EPC}
In this section, we recall the energy-preserving collocation methods for Hamiltonian systems \eqref{H-ODE}. Denoting $f(x)=S(x)\nabla H(x)$, the methods are defined as follows:
\begin{definition}\label{Def1-coll-polynomial}
\cite{hairer2010energy} Let $c_1, c_2,\cdots, c_s$ be distinct real numbers (usually $0\le c_i\le 1$) and define $\ell_i(\tau)=\prod\limits_{j=1, j \neq i}^s \frac{\tau-c_j}{c_i-c_j}, \quad b_i=\int_0^1 \ell_i(\tau) \mathrm{d} \tau$. Consider a polynomial $u(t)$ of degree $s$ satisfying
$$
\begin{aligned}
& u\left(t_0\right)=x_0 \\
& \dot{u}\left(t_0+c_i h\right)=\frac{1}{b_i} \int_0^1 \ell_i(\tau) f\left(u\left(t_0+\tau h\right)\right) \mathrm{d} \tau .
\end{aligned}
$$
The numerical solution after one step is then defined by $x_1=u\left(t_0+h\right)$.  
\end{definition}

By interpolating the function $\dot{u}\left(t_0+ \tau h\right)$ as a polynomial and using a continuum of stages $\tau\in [0,1]$, the above defined energy-preserving collocation methods can be written as Runge-Kutta methods 
\begin{equation}\label{EP-CRK}
X_\tau=x_0+h \int_0^1 A_{\tau, \sigma} f\left(X_\sigma\right) \mathrm{d} \sigma, \quad x_1=x_0+h \int_0^1 B_\sigma f\left(X_\sigma\right) \mathrm{d} \sigma,
\end{equation}
where $B_\sigma=1$ and $X_\tau$ are the values of the interpolation polynomial $u\left(t_0+\tau h\right)$ approximating 
 $x(t_0+C_\tau h)$, with $C_\tau=\int_0^1A_{\tau,\sigma}d\sigma$,  satisfying $C_\tau=\tau$. The coefficients $A_{\tau,\sigma}$ are given by
\begin{equation*}
 A_{\tau, \sigma}=\sum_{i=1}^s \frac{1}{b_i} \int_0^\tau \ell_i(\alpha) \mathrm{d} \alpha \ell_i(\sigma).
\end{equation*}
Different choices of collocation points may lead to different $A_{\tau,\sigma}$, resulting in numerical schemes of different orders. The order theorem of scheme \eqref{EP-CRK} is provided in  \cite{hairer2010energy} (see Theorem 1 therein). This order theorem 
determines the form of $A_{\tau,\sigma}$ that yields the highest order scheme. 

The averaged vector field method can be viewed as a special case of $s=1$, where $\dot{u}\left(t_0+ \tau h\right)$ is a constant, regardless of what $c_1$ is.

For $s=2$, any nodes providing a quadrature formula of order three or higher result in
Runge-Kutta coefficients:
\begin{equation*}
 A_{\tau, \sigma}=-6\tau(1-\tau)\sigma+\tau(4-3\tau), \quad B_\sigma=1.
\end{equation*}
 Consider the second polynomial $u(t)$ as a linear combination of $x_0=u(t_0)$, $x_{\frac{1}{2}}=u(t_0+\frac{h}{2})$, and $x_1=u(t_0+h)$, then the following energy-preserving collocation scheme
\begin{equation*}
\begin{split}
X_{\frac{1}{2}} & =x_0+h \int_0^1\left(\frac{5}{4}-\frac{3}{2} \sigma\right) f\left(u\left(t_0+\sigma h\right)\right) \mathrm{d} \sigma \\
x_1 & =x_0+h \int_0^1 f\left(u\left(t_0+\sigma h\right)\right) \mathrm{d} \sigma
\end{split}
\end{equation*}
is of order 4.

For $\mathbf{s}=\mathbf{3}$, $A_{\tau, \sigma}$ is a second-order polynomial of $\sigma$ with the form $A_{\tau, \sigma}=\tau(p_0(\tau)+p_1(\tau)\sigma+p_2(\tau)\sigma^2)$, where $p_0(\tau)$, $p_1(\tau)$ and $p_2(\tau)$ are second-order polynomials. Using the order condition $\int_0^1 A_{\tau, \sigma}\sigma^{k-1}d\sigma=\frac{1}{k}\tau^k$, for $k+1,\cdots, s$, we can determine $p_i(\tau), i=0,1,2$ to get the following Runge-Kutta coefficients:
$$
A_{\tau, \sigma}=\tau\left(\left(9-18 \tau+10 \tau^2\right)-12\left(3-8 \tau+5 \tau^2\right) \sigma+30\left(1-3 \tau+2 \tau^2\right) \sigma^2\right).
$$
Representing the polynomial $u(t)$ of degree 3 as linear combination of $x_0, X_{\frac{1}{3}}, X_{\frac{2}{3}}$, and $x_1$, the method reads
$$
\begin{aligned}
X_{\frac{1}{3}} & =x_0+h \int_0^1\left(\frac{37}{27}-\frac{32}{9} \sigma+\frac{20}{9} \sigma^2\right) f\left(u\left(t_0+\sigma h\right)\right) \mathrm{d} \sigma \\
X_{\frac{2}{3}} & =x_0+h \int_0^1\left(\frac{26}{27}+\frac{8}{9} \sigma-\frac{20}{9} \sigma^2\right) f\left(u\left(t_0+\sigma h\right)\right) \mathrm{d} \sigma \\
x_1 & =x_0+h \int_0^1 f\left(u\left(t_0+\sigma h\right)\right) \mathrm{d} \sigma
\end{aligned}
$$
and is of order 6.


\section{High-order exponential energy dissipation-preserving  integrators}\label{Higher-order EPC-damped}
In this section, we introduce  exponential transformations to the variables $x_0$, $x_1$ and the intermediate stages introduced in Section \ref{EPC} to construct high-order  methods preserving the energy dissipation rate for the damped Hamiltonian systems  \eqref{H-damp-ODE}. 

Denote $Y(t)=\int_{t_0+h/2}^{t} D(\tau) d\tau$ and replace $X_\sigma$ by $e^{Y_\sigma}X_\sigma$ in equation \eqref{EP-CRK}. This yields the following exponential dissipation-preserving numerical scheme for the damped system \eqref{H-damp-ODE}:
\begin{equation}\label{EEP-C}
\begin{split}
x_1&=e^{-Y_1}\big(e^{Y_0}x_0+h \int_0^1  f\left(e^{Y_\sigma}X_\sigma\right) \mathrm{d} \sigma\big),\quad \text{where}\\
X_\tau&=e^{-Y_\tau}\big(e^{Y_0}x_0+h \int_0^1 A_{\tau, \sigma} f\left(e^{Y_\sigma}X_\sigma\right) \mathrm{d} \sigma\big). 
\end{split}
\end{equation}
Similar to how the energy-preserving collocation method  \eqref{EP-CRK} is derived from Definition \ref{Def1-coll-polynomial}, we give the following  analogy  definition to derive the exponential energy-preserving scheme  \eqref{EEP-C}.   
\begin{definition}\label{Def2-coll-polynomial-modify}
Let $c_1, c_2,\cdots, c_s$ be distinct real numbers (usually $0\le c_i\le 1$) and define $\ell_i(\tau)=\prod\limits_{j=1, j \neq i}^s \frac{\tau-c_j}{c_i-c_j}, \quad b_i=\int_0^1 \ell_i(\tau) \mathrm{d} \tau$. We consider a polynomial $v(t)$ of degree $s$ satisfying 
$$
\begin{aligned}
& v\left(t_0\right)=e^{Y_0}x_0 \\
& \dot{v}\left(t_0+c_i h\right)=\frac{1}{b_i} \int_0^1 \ell_i(\tau) S\nabla H\left(v(t_0+\tau h)\right) \mathrm{d} \tau.
\end{aligned}
$$
The numerical solution after one step is then defined by $x_1=e^{-Y_1}v\left(t_0+h\right)$.  
\end{definition}
\begin{remark}\label{remark-EP-damped-def}
Notice that $X_{\tau}=e^{-Y_{\tau}}v\left(t_0+\tau h\right)$; we use $v\left(t_0+\tau h\right)=e^{Y_{\tau}}X_{\tau}$ in implementation of the proposed methods in Definition \ref{Def2-coll-polynomial-modify}.
\end{remark}
The second order method is given by the mean value averaged exponential discrete gradient method introduced in \cite{moore2021exponential}
\begin{equation*}
\begin{split}
x_1 & = e^{-Y_1}\big(e^{Y_0}x_0+h \int_0^1 f\left(v\left(t_0+\sigma h\right)\right) \mathrm{d} \sigma\big),
\end{split}
\end{equation*}
where $v(t)$ is a first order polynomial passing through the points $e^{Y_0}x_0$ and $e^{Y_1}x_1$.

The fourth order method is given by 
\begin{equation*}
\begin{split}
X_{\frac{1}{2}} & = e^{-Y_\frac{1}{2}}\big(e^{Y_0}x_0+h \int_0^1\left(\frac{5}{4}-\frac{3}{2} \sigma\right) f\left(v\left(t_0+\sigma h\right)\right) \mathrm{d} \sigma\big) \\
x_1 & = e^{-Y_1}\big(e^{Y_0}x_0+h \int_0^1 f\left(v\left(t_0+\sigma h\right)\right) \mathrm{d} \sigma\big),
\end{split}
\end{equation*}
where $v(t)$ is a second polynomial passing through the points $e^{Y_0}x_0$, $e^{Y_\frac{1}{2}}X_{\frac{1}{2}}$, and $e^{Y_1}x_1$. 

The sixth order method is given by 
$$
\begin{aligned}
X_{\frac{1}{3}} & = e^{-Y_\frac{1}{3}}\big( e^{Y_0}x_0+h \int_0^1\left(\frac{37}{27}-\frac{32}{9} \sigma+\frac{20}{9} \sigma^2\right) f\left(v\left(t_0+\sigma h\right)\right) \mathrm{d} \sigma\big) \\
X_{\frac{2}{3}} & = e^{-Y_\frac{2}{3}}\big(e^{Y_0}x_0+h \int_0^1\left(\frac{26}{27}+\frac{8}{9} \sigma-\frac{20}{9} \sigma^2\right) f\left(v\left(t_0+\sigma h\right)\right) \mathrm{d} \sigma\big) \\
x_1 & = e^{-Y_1}\big(e^{Y_0}x_0+h \int_0^1 f\left(v\left(t_0+\sigma h\right)\right) \mathrm{d} \sigma\big),
\end{aligned}
$$
where $v(t)$ is a third-order polynomial passing through the points $e^{Y_0}x_0$, $e^{Y_\frac{1}{3}}X_{\frac{1}{3}}$, $e^{Y_\frac{2}{3}}X_{\frac{2}{3}}$, and $e^{Y_1}x_1$. 

Following the order condition, the eighth-order Runge-Kutta coefficients with $\mathbf{s}=\mathbf{4}$ are
\begin{equation*}
\begin{split}
A_{\tau, \sigma}=&\left(16\tau-60 \tau^2+80 \tau^3-35\tau^4\right)+\left(-120\tau+600\tau^2-900 \tau^3+420\tau^4\right) \sigma\\
&+\left(240\tau-1350 \tau^2+2160 \tau^3-1050\tau^4\right) \sigma^2\\
&+\left(-140\tau+840 \tau^2-1400 \tau^3+700\tau^4\right)\sigma^3,
\end{split}
\end{equation*}
and the scheme has the following form
\begin{equation}
\begin{split}
X_{\frac{1}{4}} & =e^{-Y_\frac{1}{4}}\big( e^{Y_0}x_0+h \int_0^1\left(\frac{349}{256}-\frac{315}{64} \sigma+\frac{675}{128} \sigma^2-\frac{105}{64} \sigma^3\right) f\left(v\left(t_0+\sigma h\right)\right) \mathrm{d} \sigma\big) \\
X_{\frac{1}{2}} & =e^{-Y_\frac{1}{2}}\big( e^{Y_0}x_0+h \int_0^1\left(\frac{13}{16}+\frac{15}{4} \sigma-\frac{105}{8} \sigma^2+\frac{35}{4} \sigma^3\right) f\left(v\left(t_0+\sigma h\right)\right) \mathrm{d} \sigma\big) \\
X_{\frac{3}{4}} & =e^{-Y_\frac{3}{4}}\big( e^{Y_0}x_0+h \int_0^1\left(\frac{237}{256}+\frac{45}{64} \sigma-\frac{45}{128} \sigma^2-\frac{105}{64} \sigma^3\right) f\left(v\left(t_0+\sigma h\right)\right) \mathrm{d} \sigma\big) \\
x_1 & =e^{-Y_1}\big( e^{Y_0}x_0+h \int_0^1 f\left(v\left(t_0+\sigma h\right)\right) \mathrm{d} \sigma\big),
\end{split}
\end{equation}
where the fourth-order polynomial $v(t)$  passes through the points
$e^{Y_0}x_0$, $e^{Y_\frac{1}{4}}X_{\frac{1}{4}}$, $e^{Y_\frac{1}{2}}X_{\frac{1}{2}}$, $e^{Y_\frac{3}{4}}X_{\frac{3}{4}}$, and $e^{Y_1}x_1$. 

\begin{theorem}\label{Structure-preservation}
For system \eqref{H-damp-ODE}, the exponential  energy-preserving collocation methods \eqref{EEP-C} preserve the dissipative rate, i.e., 
$\frac{dH(x)}{dt}=-\eta(t)H(x)$, 
whenever $H$ satisfies 
\begin{equation}\label{constraint}
H(e^{Y_t}x_t)=e^{\phi_t}H(x_t), 
\end{equation}
with $\phi_t=\int_{t_0+h/2}^{t_0+(t-t_0)}\eta(s)ds$.
\end{theorem}
\begin{proof}
By the definition of energy-preserving collocation methods in \ref{Def2-coll-polynomial-modify}, we obtain 
\begin{equation*}
\begin{split}
&H(e^{Y_1}x_1)-H(e^{Y_0}x_0)\\
&=H\left(v\left(t_0+h\right)\right)-H\left(v\left(t_0\right)\right)\\
&=h \int_0^1 \dot{v}\left(t_0+\tau h\right)^{\top} \nabla H\left(v\left(t_0+\tau h\right)\right) \mathrm{d} \tau \\
&=h \sum_{i=1}^s \frac{1}{b_i} \int_0^1 \ell_i(\tau) \nabla H\left(v\left(t_0+\tau h\right)\right)^{\top} \mathrm{d} \tau S^{\top} \int_0^1 \ell_i(\tau) \nabla H\left(v\left(t_0+\tau h\right)\right) \mathrm{d} \tau\\
&=0.
\end{split}
\end{equation*}
This implies 
\begin{equation*}
H(x_1)=e^{\phi_{t_0}-\phi_{t_1}}H(x_0),
\end{equation*}
whenever \eqref{constraint} is satisfied.

\end{proof}

\begin{remark}
The conditions in the above energy preservation theorem are not trivial, similarly as demonstrated in \cite{moore2021exponential}. For a general energy and matrix $D(t)$, the proposed exponential  method preserves  the energy along a transformed solution, i.e., 
$H(e^{Y_1}x_1)=H(e^{Y_0}x_0)$, which is equivalent to 
\begin{equation*}\label{EP-law-num}
\frac{d}{d t} H(e^{\int_{t_0+h/2}^{t} D(s) d s} x(t))=(\nabla H(e^{ \int_{t_0+h/2}^{t} D(s) d s} x(t)))^T e^{\int_{t_0+h/2}^{t} D(s) d s}(\dot{x}+D(t) x)=0.
\end{equation*}
But the true energy conservation property follows 
\begin{equation*}\label{EP-law-true}
(\nabla H(x(t)))^T(\dot{x}+D(t) x)=0.
\end{equation*}
Although for general case the the proposed method dose not preserve  the true energy conservatin property, it keeps a nice energy balance property that is close to the exact one.
\end{remark}

\begin{remark}
If an energy-preserving collocation method \eqref{EP-CRK} for Hamiltonian system \eqref{H-ODE} preserves another invariant $I(x)$ other than the energy function $H(x)$, the corresponding exponential  energy-preserving collocation method also preserves the structure $\frac{dI(x)}{dt}=-\eta(t)I(x)$, 
as long as  $I$ satisfies $I(e^{Y_t}x_t)=e^{\phi_t}I(x_t)$.
\end{remark}

\begin{remark}\label{remark-S}
If the skew-symmetric matrix $S$ depends on time, e.g., $S=S(t,x)$, the above proposed methods still work, but we need to consider a discretization of $S(t,x)$ to keep its skew symmetry. For example, we can discretize it as $S(\frac{t_0+t_{1}}{2},\frac{e^{Y_0}x_0+e^{Y_1}x_{1}}{2})$.
\end{remark}

\begin{theorem}\label{EEP-C-symmetric}
The exponential  energy-preserving collocation method \eqref{EEP-C} is symmetric when the collocation points are symmetric, i.e., $c_{s+1-i}=1-c_i$, for all $i=1,\cdots,s$.
\end{theorem}

\begin{proof}
By exchanging $t_0, y_0$ and $t_1, y_1$ and replacing $h$ by $-h$ in \eqref{EEP-C}, we obtain 
\begin{subequations}
\label{EEP-C-c}
 \begin{align}
  x_0&=e^{-Y_0}\big(e^{Y_1}x_1-h \int_0^1  f\left(e^{Y_{1-\sigma}}X_{1-\sigma}\right) \mathrm{d} \sigma\big)\label{eq11} \\
  X_{1-\tau}&=e^{-Y_{1-\tau}}\big(e^{Y_1}x_1-h \int_0^1 A_{\tau, \sigma} f\left(e^{Y_{1-\sigma}}X_{1-\sigma}\right) \mathrm{d} \sigma\big).\label{eq12}
 \end{align}
\end{subequations}
Multiplying both sides of equation \eqref{eq11} by $e^{Y_0}$ and both sides of \eqref{eq12} by $e^{Y_{1-\tau}}$, we obtain 
\begin{subequations}
\label{EEP-C-sym}
 \begin{align}
  e^{Y_0}x_0&=e^{Y_1}x_1-h \int_0^1  f\left(e^{Y_{1-\sigma}}X_{1-\sigma}\right) \mathrm{d} \sigma\label{sym-eq11} \\
  e^{Y_{1-\tau}}X_{1-\tau}&=e^{Y_1}x_1-h \int_0^1 A_{\tau, \sigma} f\left(e^{Y_{1-\sigma}}X_{1-\sigma}\right) \mathrm{d} \sigma.\label{sym-eq12}
 \end{align}
 \end{subequations}
Combining equations \eqref{sym-eq11} and \eqref{sym-eq12}, we obtain
 \begin{equation}\label{sym-eq2}
  e^{Y_{1-\tau}}X_{1-\tau}=e^{Y_0}x_0+h \int_0^1\big(1-A_{\tau, \sigma}\big) f\left(e^{Y_{1-\sigma}}X_{1-\sigma}\right) \mathrm{d} \sigma.
 \end{equation}
 Since symmetric nodes $c_1,\cdots, c_s$ guarantees $A_{\tau, \sigma}+A_{1-\tau, 1-\sigma}=1$, the above equation \eqref{sym-eq2} can be rewritten as
  \begin{equation}\label{sym-eq3}
  \begin{split}
  e^{Y_{1-\tau}}X_{1-\tau}&=e^{Y_0}x_0+h \int_0^1 A_{1-\tau, 1-\sigma} f\left(e^{Y_{1-\sigma}}X_{1-\sigma}\right) \mathrm{d} \sigma\\
  &=e^{Y_0}x_0+h \int_0^1 A_{1-\tau,\sigma} f\left(e^{Y_{\sigma}}X_{\sigma}\right) \mathrm{d} \sigma.
  \end{split}
 \end{equation}
 Equations \eqref{sym-eq3} and \eqref{sym-eq11} indicate that scheme \eqref{EEP-C} is  symmetric when symmetric nodes are utilized.

\end{proof}

\section{Numerical experiments}\label{Numerical-test}
Previous studies have highlighted the advantage of exponential discrete gradient methods over other structure-preserving methods such as the standard discrete gradient method, symplectic methods (e.g., the implicit midpoint method), and conformal symplectic methods (e.g., the implicit conformal symplectic method). Therefore, our numerical experiments primarily aim to validate the high order and the conservation of energy structure of our proposed method. In this section, we apply the proposed Exponential Energy dissipation-Preserving Collocation methods  (EEPC) of different orders to Hamiltonian differential equations including damped Burger's equation and KdV equation, with different damping terms.  The solution of the eighth-order method serves as a reference solution to test the accuracy of the proposed methods.

For the damped Hamiltonian system \eqref{H-damp-ODE}, a conformal invariant $I$ (including $H$) satisfying the assumption in Theorem \ref{Structure-preservation}  preserves the dissipation rate
$$
I(x(t))=e^{-\int_0^{t}\eta(s)ds} I\left(x(0)\right),
$$
where $\eta(s)$ is positive, ensuring that the exponential integrator maintains the dissipation of this function at each time step, i.e.
$$
I\left(x_{n+1}\right)<I\left(x_n\right).
$$
If the form of $\eta(s)$ for a conformal invariant is known, the numerical dissipative rate can be evaluated using
$$
R_I=\ln \left(\frac{I\left(x_{n+1}\right)}{I\left(x_{n}\right)}\right)+\int_{t_n}^{t_{n+1}}\eta(s) ds,
$$
where $R_I$ should be machine accurate. However, in most cases, the exact form of $\eta(s)$ is unknown. Instead, we use the following residual,  similar to that defined in \cite{moore2021exponential}, to measure how well a method preserves the dissipation rate of a function $I(x)$ (including the energy function $H(x)$)
\begin{equation}\label{dissp-rate-average}
R_I=\ln \left(\frac{I\left(x_{n+1}\right)}{I\left(x_{n}\right)}\right)+\Delta t \sum_{k=1}^{k=N}\frac{\gamma_k}{N}.    
\end{equation}
$R_I$ remains a small bounded number when $\gamma_k$ (for $k=1,\cdots, N$) are small.

We represent the centered finite difference discretization of the first and second-order derivative operators $\partial_x$ and $\partial_{x x}$, respectively, under periodic boundary conditions using the matrices $D_1 \in \mathbb{R}^{N_1 \times N_1}$ and $D_2 \in \mathbb{R}^{N_1 \times N_1}$:
\begin{equation*}
\begin{split}
D_1&=\frac{1}{2 \Delta x}\left(\begin{array}{ccccc}
0 & 1 & & & -1 \\
-1 & 0 & 1 & & \\
& \ddots & \ddots & \ddots & \\
& & -1 & 0 & 1 \\
1 & & & -1 & 0
\end{array}\right),\\
D_2&=\frac{1}{\Delta x^2}\left(\begin{array}{ccccc}
-2 & 1 & & & 1 \\
1 & -2 & 1 & & \\
& \ddots & \ddots & \ddots & \\
& & 1 & -2 & 1 \\
1 & & & 1 & -2
\end{array}\right),
\end{split}
\end{equation*}
where $N_1$ is determined by the number of grids in spatial domain.
\subsection{Damped Burger's equation}\label{}
The damped Burger's equation introduced in \cite{bhatt2016second} can be written as a damped Hamiltonian PDE in the form:
\begin{equation*}\label{damped-burger}
u_t=-\frac{\partial}{\partial x} \frac{\delta \mathcal{H}}{\delta u}-2 \gamma u, \quad \mathcal{H}(u)=\frac{1}{3} \int u^3 d x.
\end{equation*}
We consider periodic boundary conditions on the spatial interval $x\in [-L, L]$ with $L=\pi$. 
Using uniform spatial grids with grid size $\Delta x$ and denoting  the semi-discrete function value as  $\mathbf{u}=[u^1, u^2, \ldots, u^{N_1}]$, 
we discretize the energy as $H\left(\mathbf{u}\right)=\frac{1}{3}\sum_{j=1}^{N_1}(u^j)^3$ 
and employ a central difference for the first-order derivative $\partial_x$. 
The resulting  semi-discrete system is
\begin{equation}\label{damped-burger-semiH}
\dot{\mathbf{u}}=-\frac{1}{2} D_1\nabla H(\mathbf{u})-2 \gamma \mathbf{u}.
\end{equation}
Notice that $\mathbf{u}^TD_1\mathbf{u}=\boldsymbol{1}^TD_1\mathbf{u}^2$ for any vector $\mathbf{u}$, where $\boldsymbol{1}$ is a vector of size $1\times N_1$ with all elements equal to  1. This indicates that $M(\mathbf{u})=\sum_{k=1}^{k=N_1} u^k$ is a conformal invariant for the semi-discrete Hamiltonian system \eqref{damped-burger-semiH}, corresponding to the discretization of the linear conformal invariant (mass) $\int u dx$ of the original problem \cite{bhatt2016second}. The proposed methods applied to \eqref{damped-burger-semiH} should precisely preserve the dissipative rate, i.e.,   $R_M=\ln \left(\frac{M\left(\mathbf{u}^{n+1}\right)}{M\left(\mathbf{u}^n\right)}\right)+2\gamma \Delta t$ should be machine accurate, where $\Delta t$ is the time step size.

We consider an Gaussian distributed initial value $u(x,0)=e^{-x^2/2}/\sqrt{2\pi}$,  $\Delta t=0.009$ and $\Delta x=\pi/40$, with the integration interval $t\in [0, T]$ and $T=50$. The damped Burger's equation \eqref{damped-burger-semiH} corresponds to equation \eqref{H-damp-ODE}, where the diagonal matrix $D(t)$ has  constant equal diagonal elements.  We also consider more general cases with unequal constant diagonal elements and time-dependent diagonal elements to evaluate the performance of our methods.

\textbf{Case I: $\textbf{D(t)}$ with constant equal diagonal elements $\gamma=0.25$.}   We observe that the proposed energy dissipation-preserving collocation methods of different orders produce similar numerical solutions. Figure \ref{solution-burgers} shows the numerical solution obtained by the second-order method.  As illustrated, the amplitude of the numerical solution diminishes over time, indicating rapid damping. Figure   \ref{energyDecay-burgers} shows the exponential decay of both the Hamiltonian $H$ and  mass $M$ over time. Figure \ref{E_decay_rate_burgers} demonstrates that the proposed methods precisely preserve the dissipation rate of mass, while the dissipation rate  of energy remains bounded and oscillatory throughout the integration. Figure   \ref{order_burgers} confirms the order of the proposed methods.
\begin{figure}[H]
     \centering
     \begin{subfigure}[b]{0.49\textwidth}
         \centering
\includegraphics[width=\textwidth]{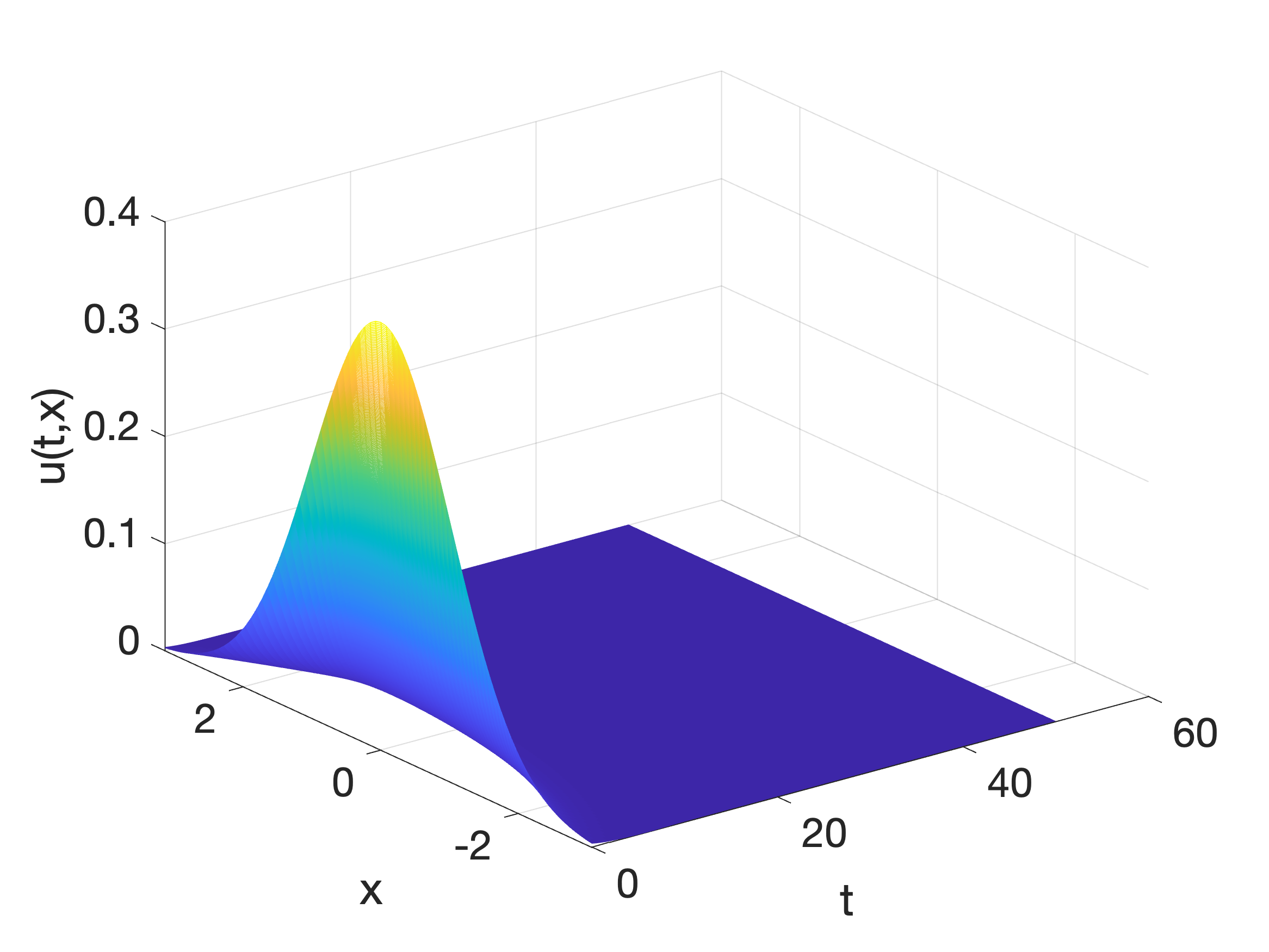}
         \caption{Solution}
         \label{solution-burgers}
     \end{subfigure}
     \hfill
       \begin{subfigure}[b]{0.49\textwidth}
         \centering
\includegraphics[width=\textwidth]{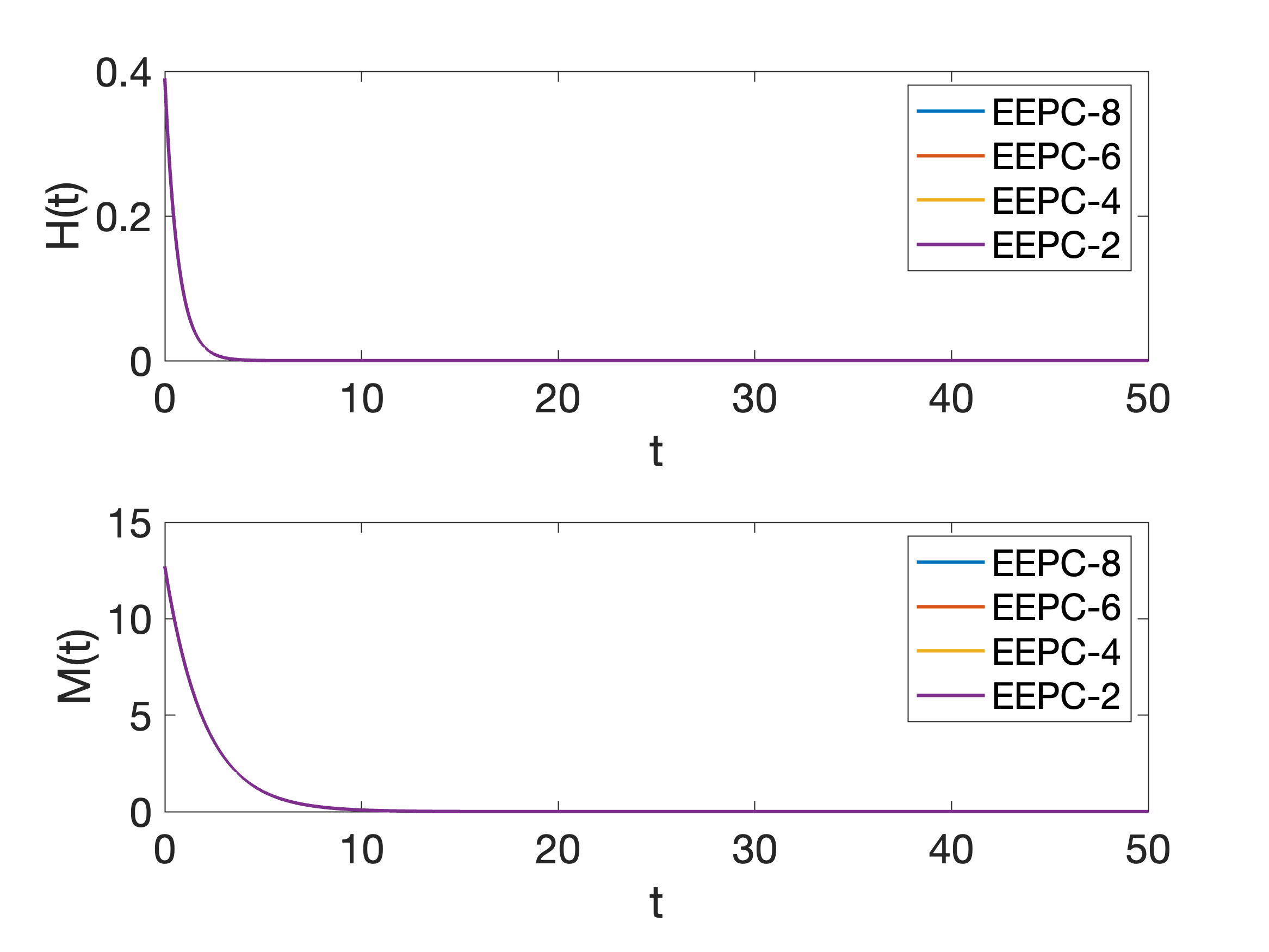}
         \caption{Decay of energy}
         \label{energyDecay-burgers}
     \end{subfigure}
       \hfill
          \begin{subfigure}[b]{0.49\textwidth}
         \centering
\includegraphics[width=\textwidth]{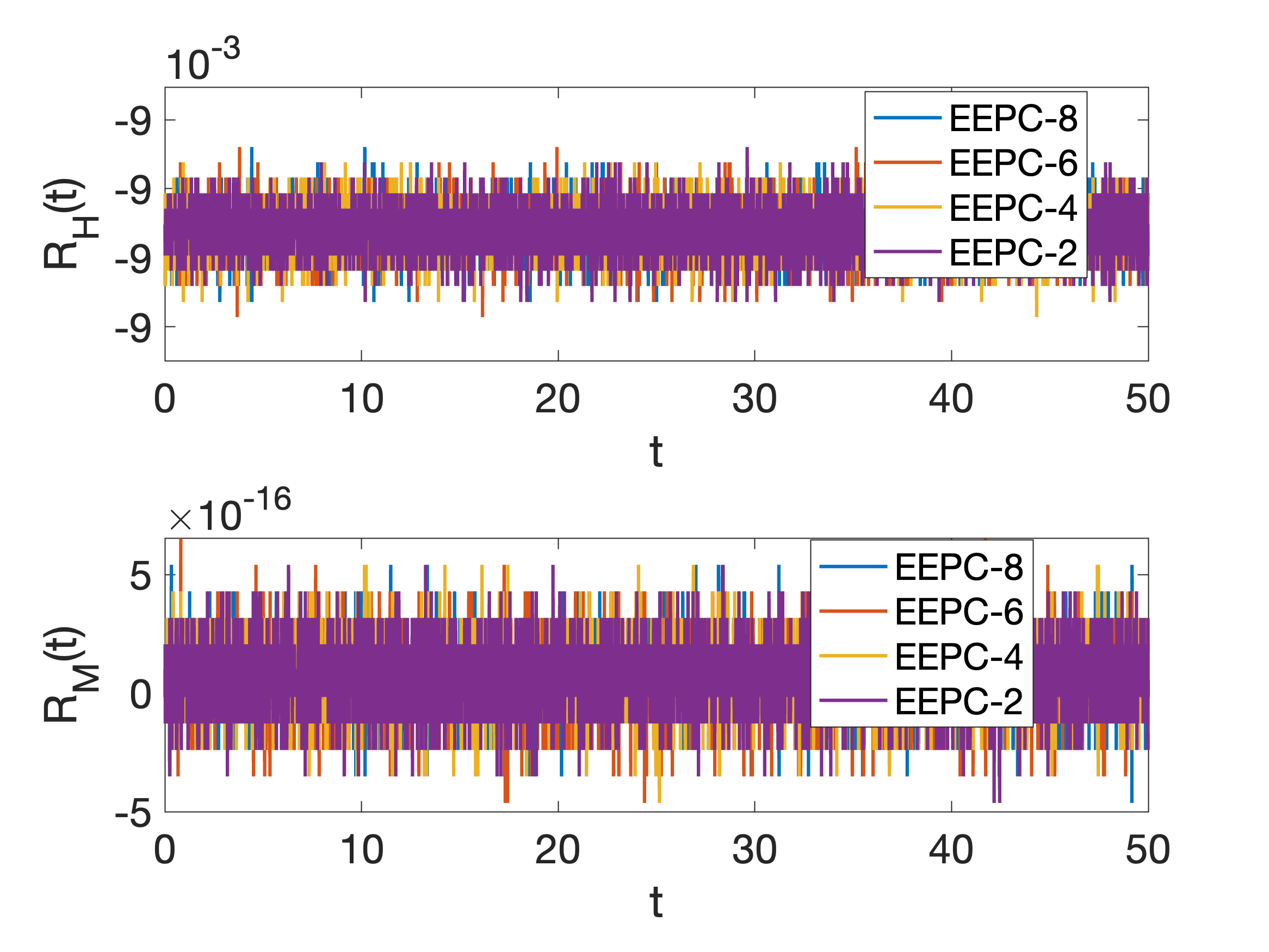}
         \caption{Conservation of energy decay rate}
         \label{E_decay_rate_burgers}
     \end{subfigure}
     \hfill
       \begin{subfigure}[b]{0.48\textwidth}
         \centering
\includegraphics[width=\textwidth]{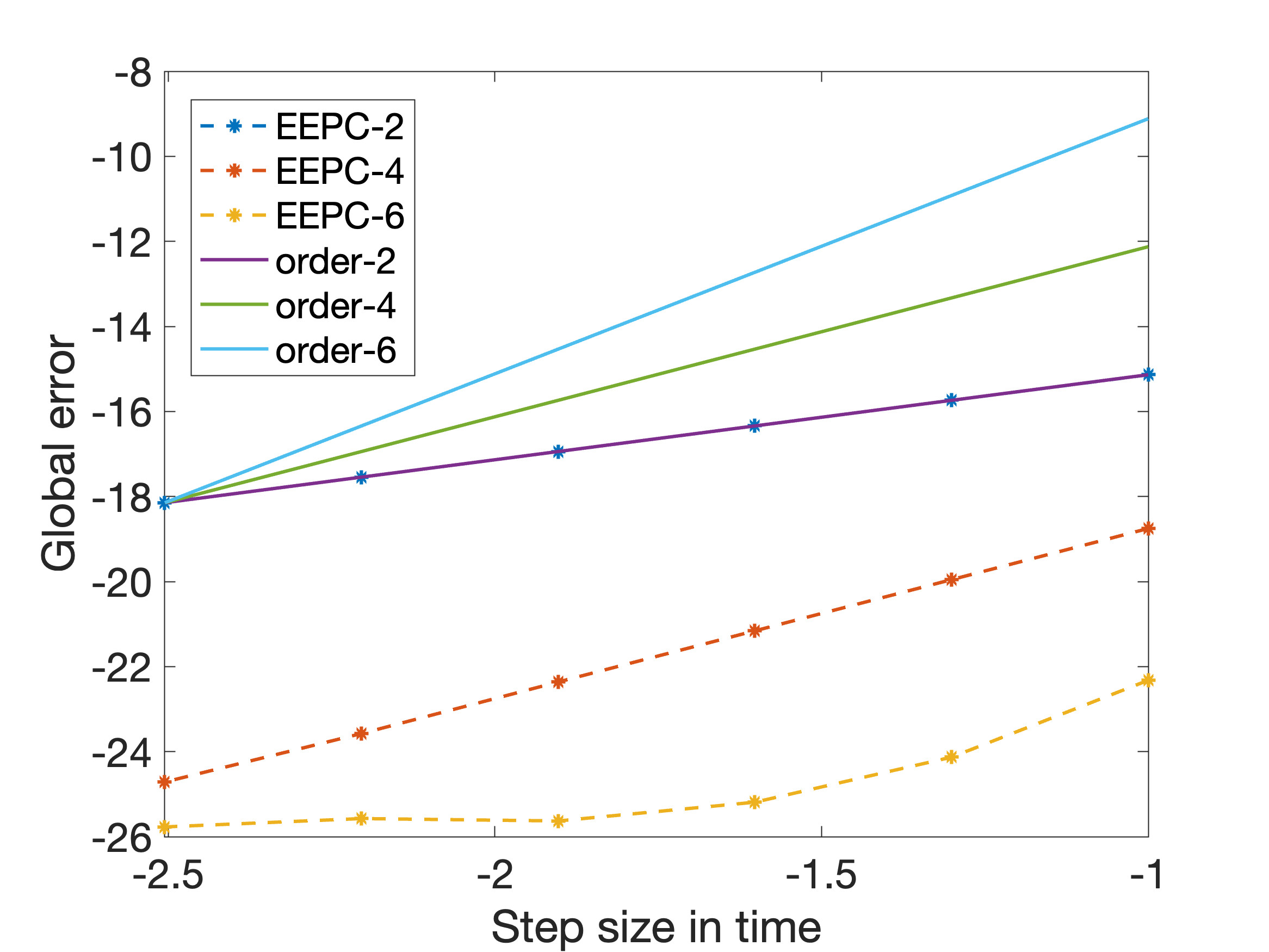}
         \caption{Order in time}
         \label{order_burgers}
     \end{subfigure}
     \hfill
        \caption{Plots of the exponential energy dissipation-preserving collocation methods of various orders for the Burger's equation \eqref{damped-burger-semiH} with $\gamma=0.25$ and $\Delta x=\pi/40$. Figures  \ref{solution-burgers}, \ref{energyDecay-burgers} and  \ref{E_decay_rate_burgers} utilize fixed time step $\Delta t=0.009$ is used.}
        \label{Burgers-gamma1}
\end{figure}

\textbf{Case II: $\textbf{D(t)}$ with constant unequal diagonal elements.}  
In this scenario, we introduce a random perturbation to each default diagonal element of $D(t)$ from Case I, within $10\%$ of the default values\footnote{Here, the random diagonal elements are $\gamma_k=0.25\times(1+2\times \text{rand}-1$) using \textit{Matlab} command.}.   Figure \ref{solution-burgers_gamma2} presents the numerical solution obtained using the proposed method\footnote{Noting that the methods of different orders produce similar numerical solutions, we only present the numerical solution by the second order method here.}. It is observed that the numerical solution exhibits reduced smoothness due to the unequal diagonal elements introduced by the random perturbations. The wave tends to damp over time, and negative values emerge during propagation. Figure   \ref{energyDecay-burgers_gamma2} shows the  decay of both energy $H$ and  mass $M$ over time. Figure \ref{E_decay_rate_burgers_gamma2} demonstrates that when the condition in Theorem  \ref{Structure-preservation} is not satisfied, the mass decay rate $R_M$ can not be  preserved. Figure   \ref{order_burgers_gamma2} confirms the order of the proposed methods.

\begin{figure}[H]
     \centering
     \begin{subfigure}[b]{0.49\textwidth}
         \centering
\includegraphics[width=\textwidth]{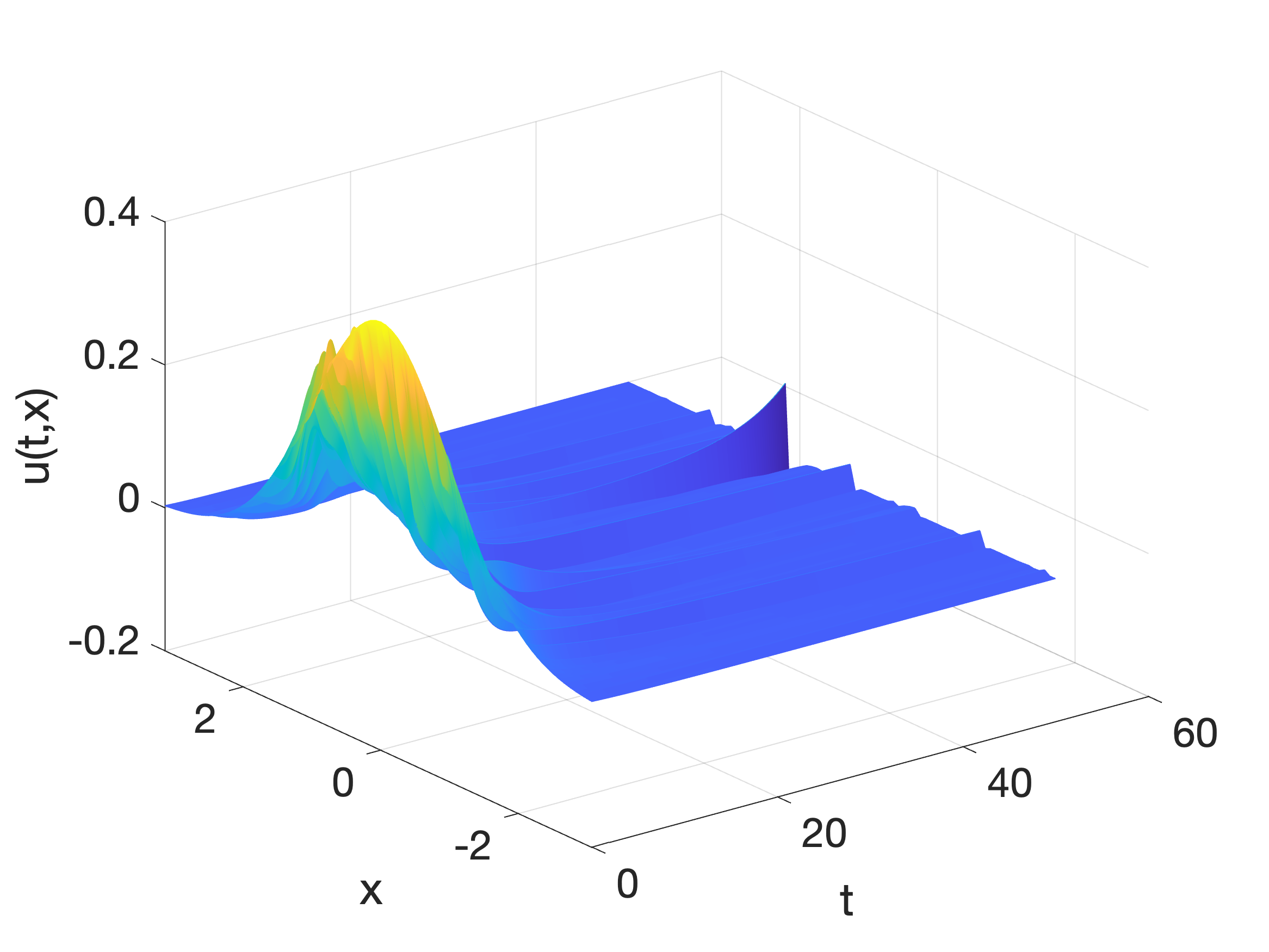}
         \caption{Solution}
         \label{solution-burgers_gamma2}
     \end{subfigure}
     \hfill
       \begin{subfigure}[b]{0.49\textwidth}
         \centering
\includegraphics[width=\textwidth]{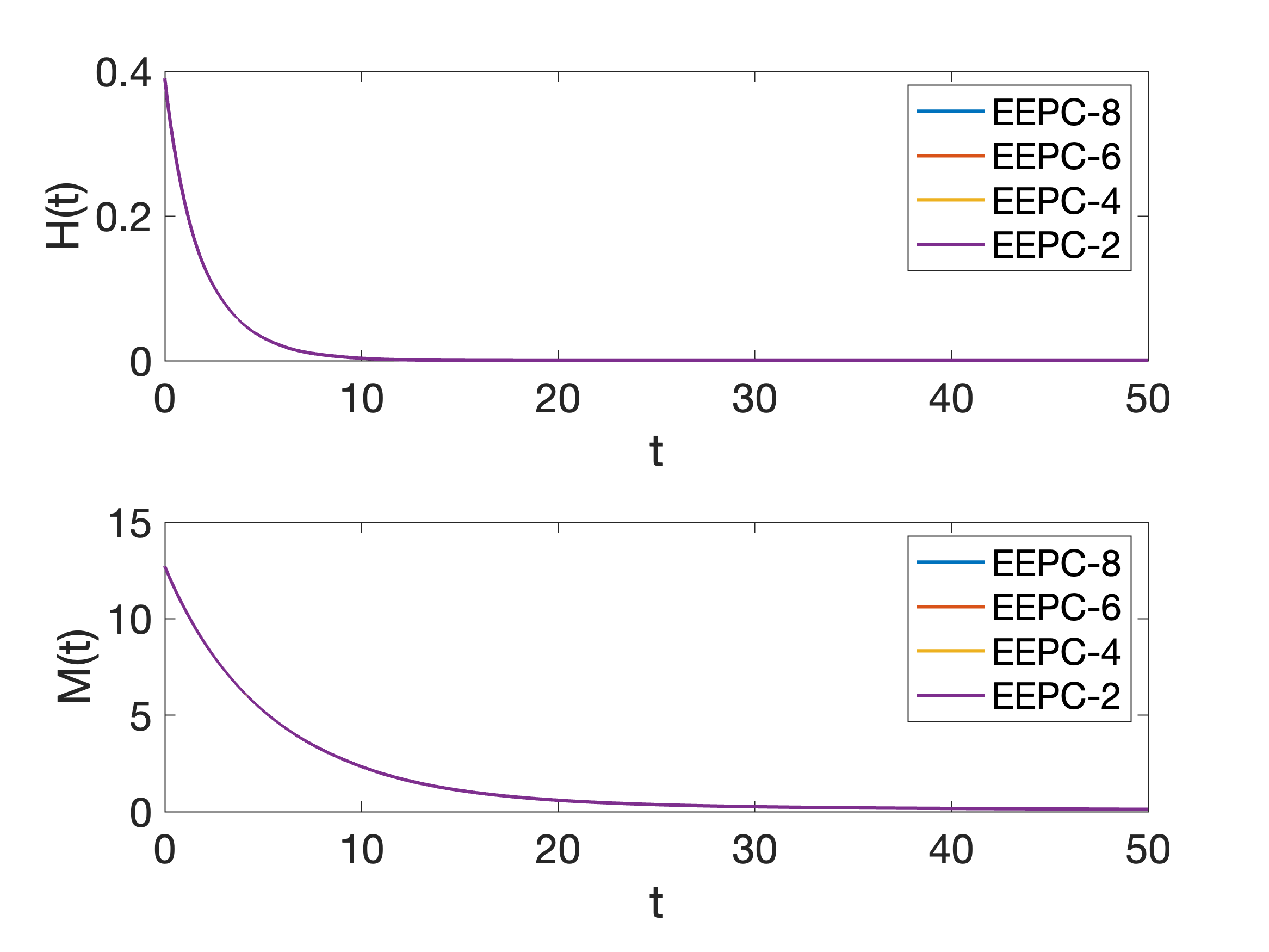}
         \caption{Decay of energy}
         \label{energyDecay-burgers_gamma2}
     \end{subfigure}
       \hfill
          \begin{subfigure}[b]{0.49\textwidth}
         \centering
\includegraphics[width=\textwidth]{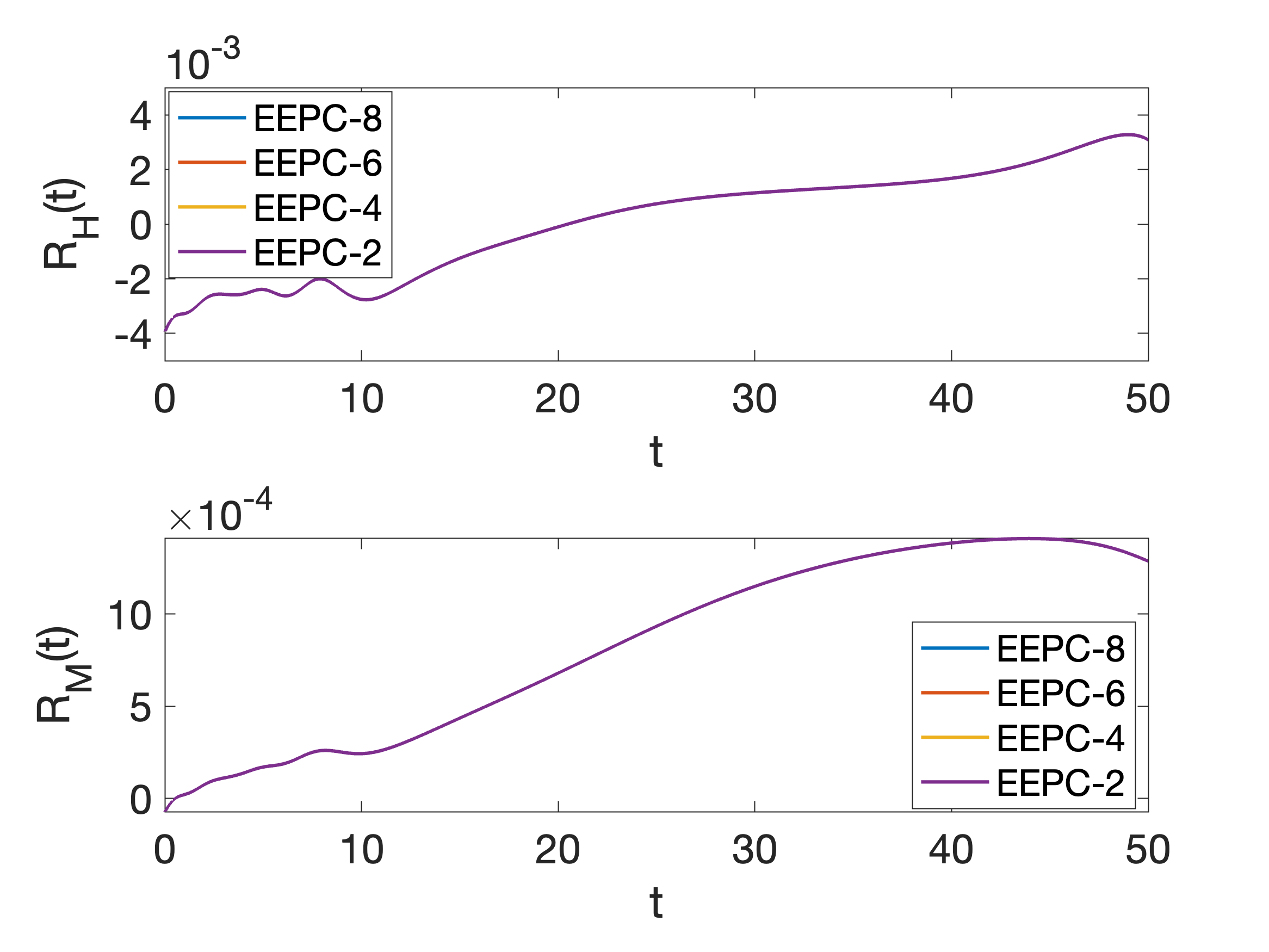}
         \caption{Conservation of energy decay rate}
         \label{E_decay_rate_burgers_gamma2}
     \end{subfigure}
     \hfill
       \begin{subfigure}[b]{0.48\textwidth}
         \centering
\includegraphics[width=\textwidth]{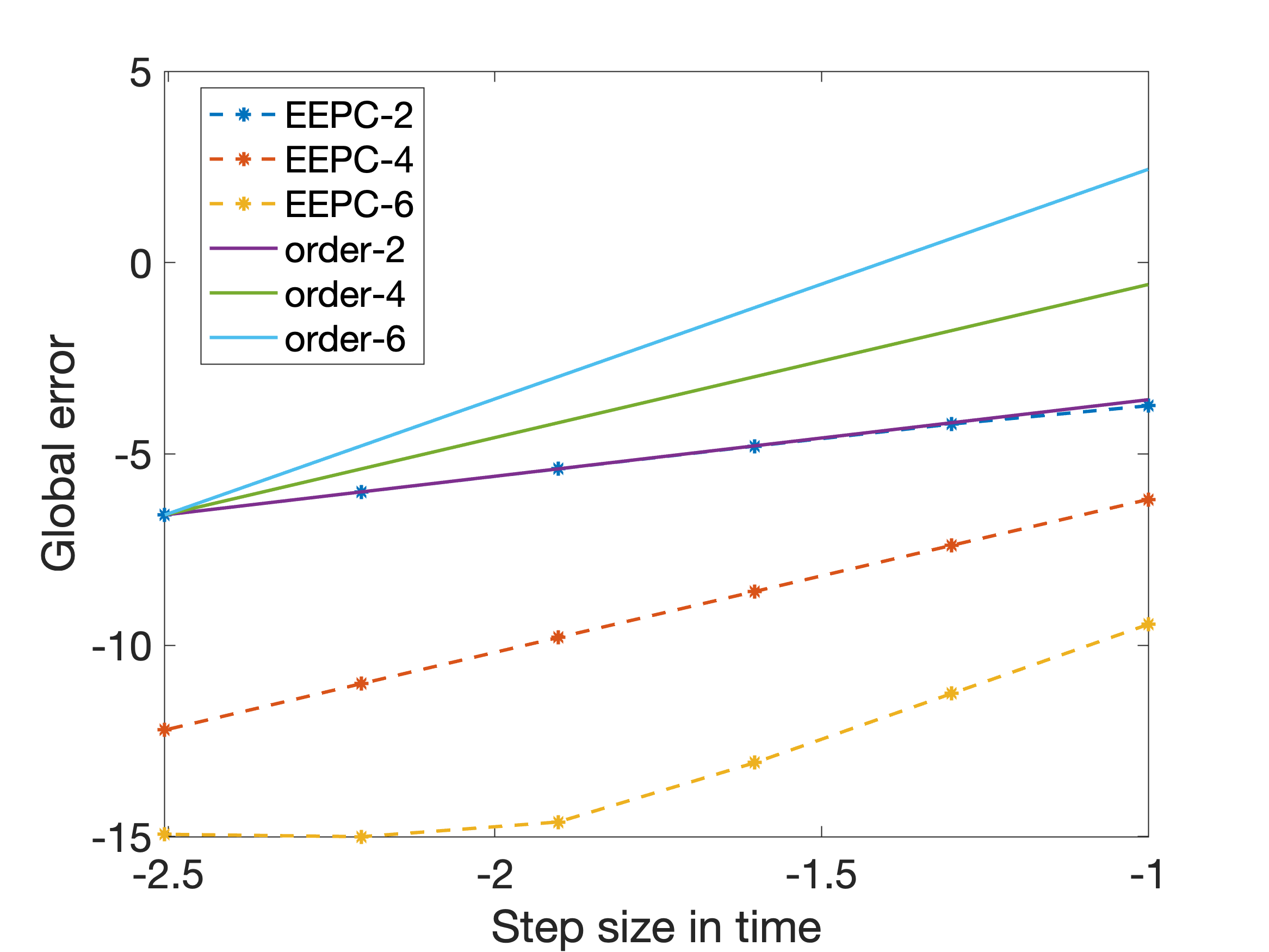}
         \caption{Order in time}
         \label{order_burgers_gamma2}
     \end{subfigure}
     \hfill
        \caption{Plots of the exponential energy dissipation-preserving collocation methods of various orders  for the Burger's equation with $D(t)$ having constant unequal diagonal elements and $\Delta x=\pi/40$. In Fig  \ref{solution-burgers_gamma2}, \ref{energyDecay-burgers_gamma2} and  \ref{E_decay_rate_burgers_gamma2}, fixed time step $\Delta t=0.009$ is used.}
        \label{Burgers-gamma2}
\end{figure}

\textbf{Case III: $\textbf{D(t)}$ with time-dependent equal diagonal elements.} We consider damped Burger's equation \ref{solution-burgers_gamma2} with damping factor $\gamma=e^{-t}$.  Figure \ref{solution-burgers_gamma3} reports the numerical solution of the second-order exponential energy-preserving collocation method, and we observe similar numerical solutions by methods of different orders. As seen in Fig  \ref{energyDecay-burgers_gamma3} and \ref{solution-burgers_gamma3}, both energy $H$ and  mass $M$ decay over time exponentially and the amplitude of the numerical wave solutions diminishes.  Figure \ref{E_decay_rate_burgers_gamma3} demonstrates that the proposed methods precisely preserve the dissipation rate of mass while the dissipation rate of energy remains bounded and oscillatory throughout the integration. Figure   \ref{order_burgers_gamma3} confirms the order of the proposed methods.

\begin{figure}[H]
     \centering
     \begin{subfigure}[b]{0.49\textwidth}
         \centering
\includegraphics[width=\textwidth]{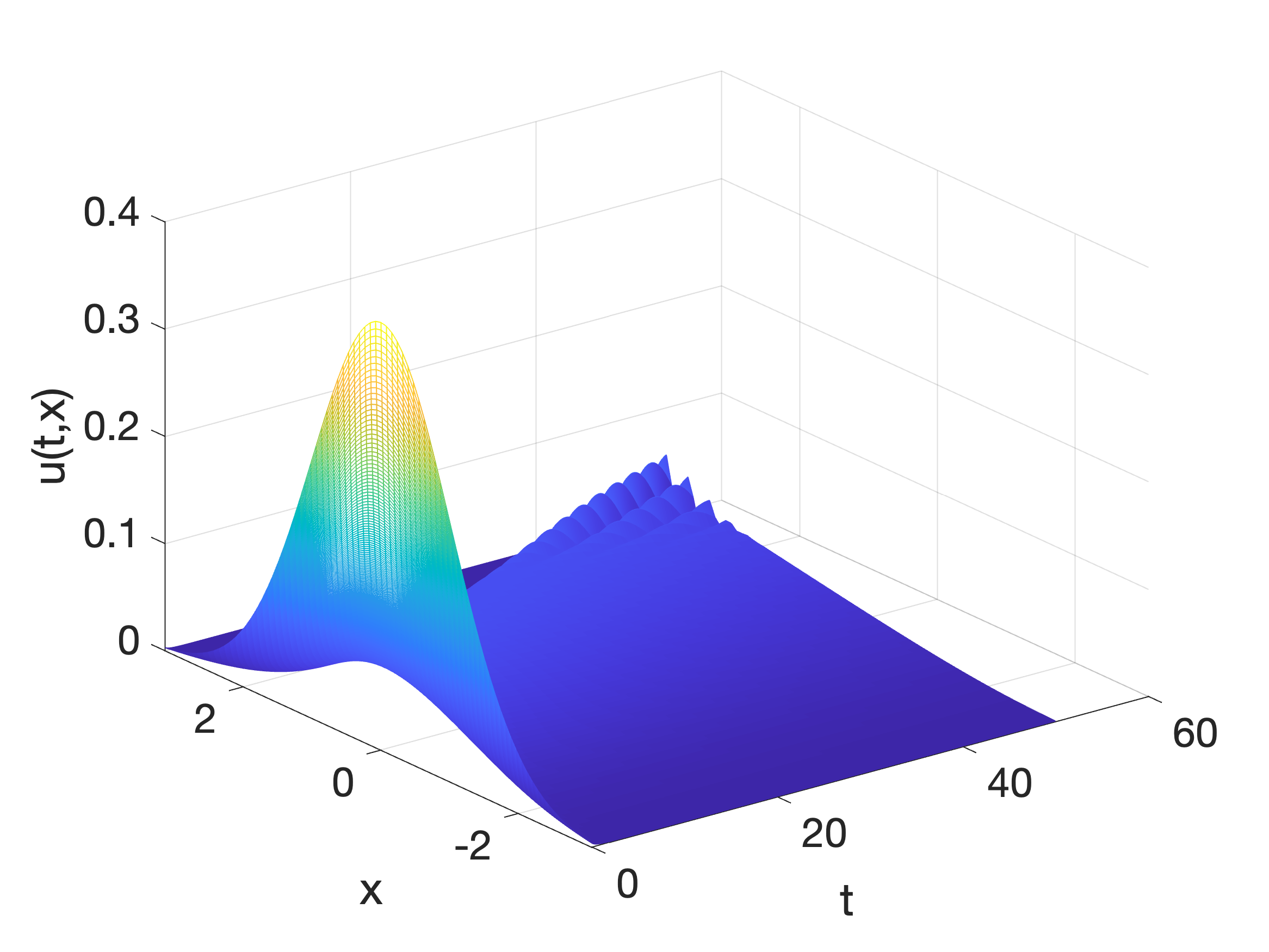}
         \caption{Solution}
         \label{solution-burgers_gamma3}
     \end{subfigure}
     \hfill
       \begin{subfigure}[b]{0.49\textwidth}
         \centering
\includegraphics[width=\textwidth]{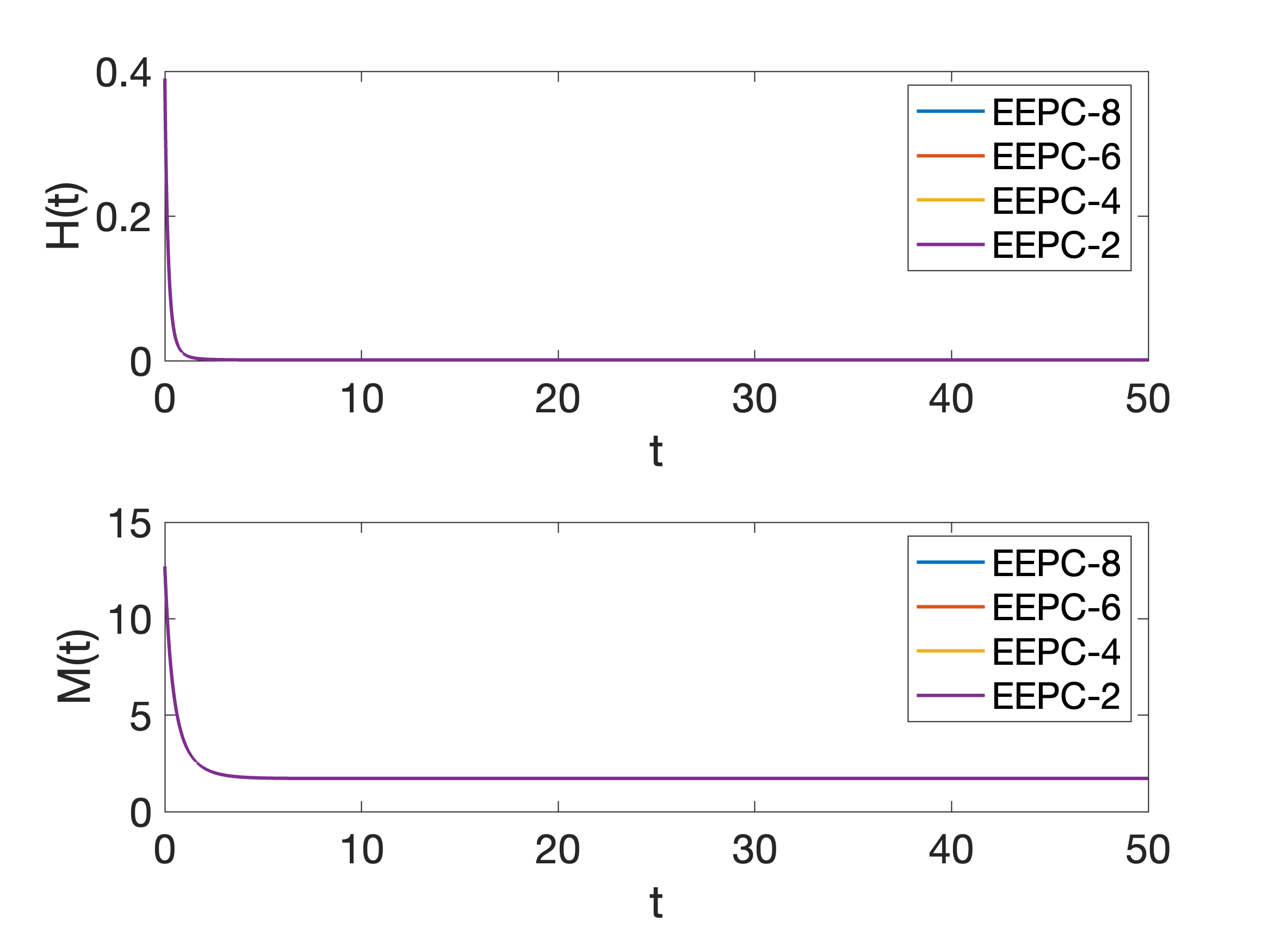}
         \caption{Decay of energy}
         \label{energyDecay-burgers_gamma3}
     \end{subfigure}
       \hfill
          \begin{subfigure}[b]{0.49\textwidth}
         \centering
\includegraphics[width=\textwidth]{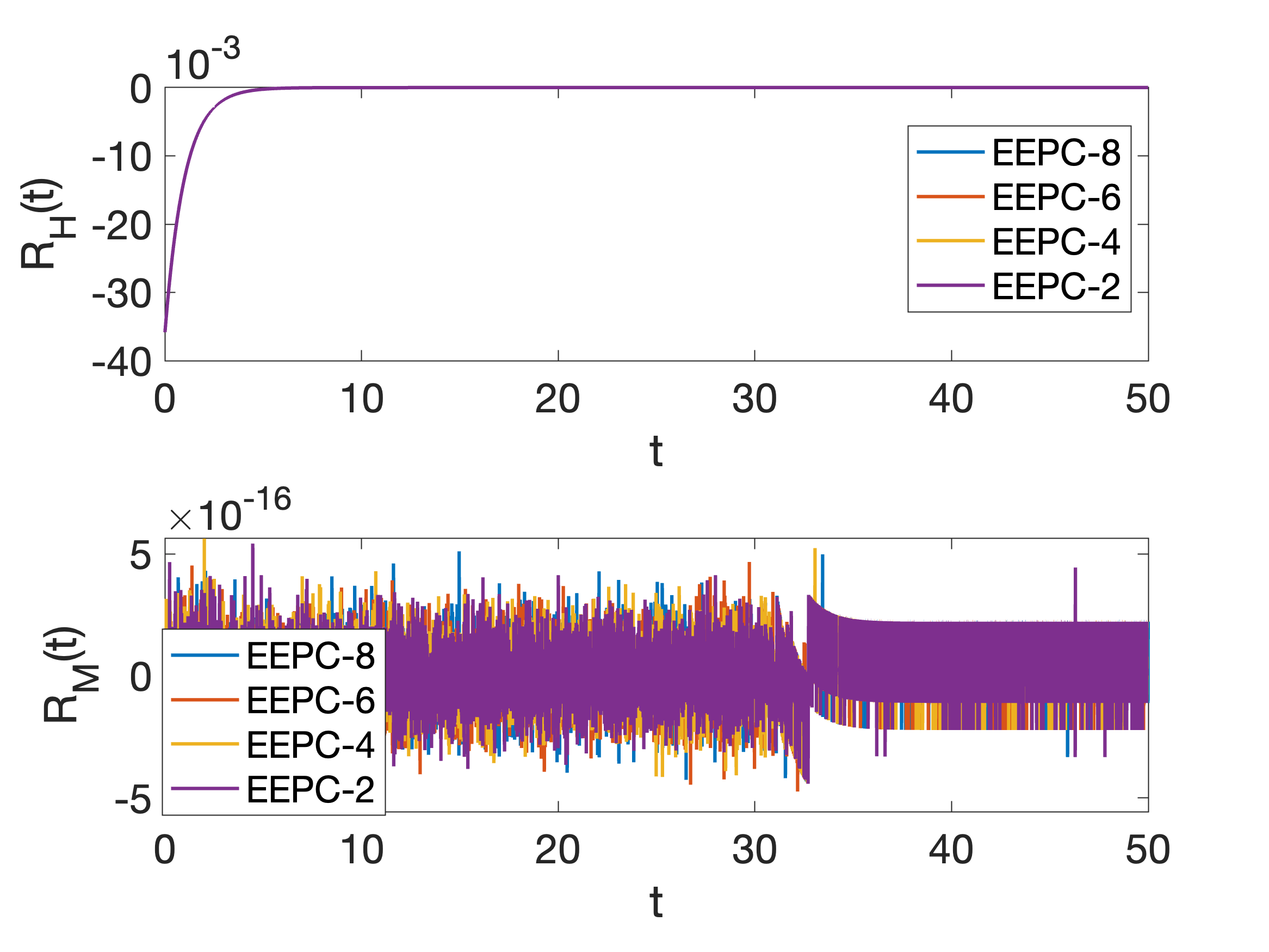}
         \caption{Conservation of energy decay rate}
         \label{E_decay_rate_burgers_gamma3}
     \end{subfigure}
     \hfill
       \begin{subfigure}[b]{0.48\textwidth}
         \centering
\includegraphics[width=\textwidth]{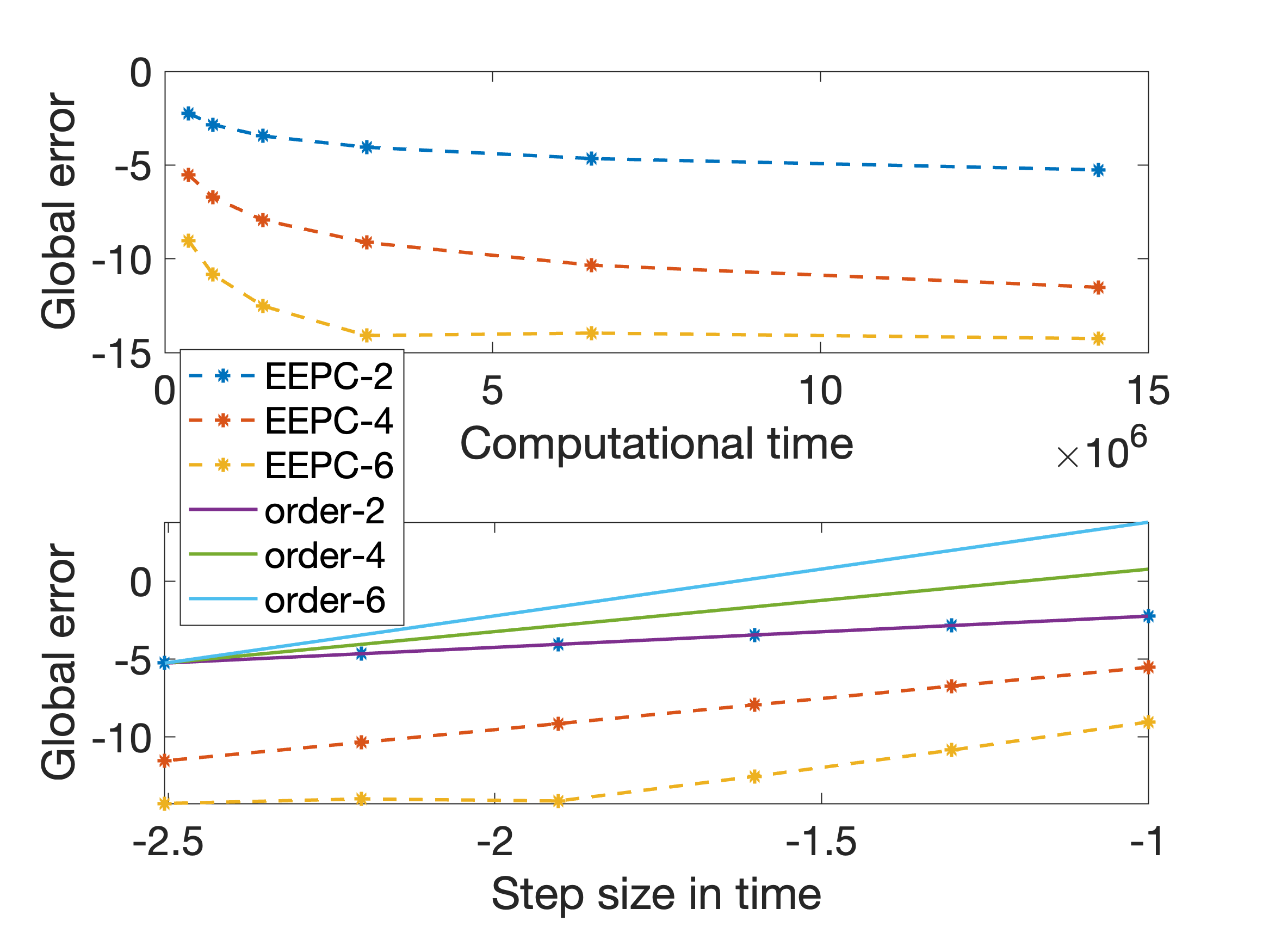}
         \caption{Order in time}
         \label{order_burgers_gamma3}
     \end{subfigure}
     \hfill
        \caption{Plots of the exponential energy dissipation-preserving collocation methods of various orders for the Burger's equation with time-dependent equal diagonal elements, and $\Delta x=\pi/40$. In Fig  \ref{solution-burgers_gamma3}, \ref{energyDecay-burgers_gamma3} and  \ref{E_decay_rate_burgers_gamma3}, fixed time step $\Delta t=0.009$ is used.}
        \label{Burgers-gamma3}
\end{figure}

\subsection{Damped KdV equation}

We consider the generalized KdV equation with Ekman damping, given by:
$$
u_t=2\alpha uu_x+\rho u_x+\nu u_{xxx}-2 \gamma u,
$$
where $\gamma, \alpha, \rho,  \nu $ are constant parameters \cite{bhatt2021projected}. This equation is often used in modeling of long-wavelength water waves and ion-acoustic waves. 
In the absence of the damping term ($\gamma=0$),  the KdV equation exhibits a bi-Hamiltonian structure 
$$
u_t=\mathcal{J}_1 \frac{\delta \mathcal{H}_1}{\delta u}=\mathcal{J}_2 \frac{\delta \mathcal{H}_2}{\delta u},
$$
with 
$$
\mathcal{J}_1=\partial_x,\quad \mathcal{H}_1=\int_{-L}^{L}\frac{\alpha}{3}u^3+\frac{\rho}{2}u^2-\frac{\nu}{2}{u_x}^2 dx,
$$
and 
$$
\mathcal{J}_2=\nu\partial_x^3+\frac{2\alpha}{3} (u \partial_x+\partial_x u)+\rho \partial_x, \quad \mathcal{H}_2=\int_{-L}^{L} \frac{1}{2} u^2 d x.
$$
Here, $\frac{\delta \mathcal{H}}{\delta u}$ denotes the functional derivative of the Hamiltonian $\mathcal{H}$ with respect to $u$. We set $\alpha=-3/8$, $\rho=-10^{-1}$, and $\nu=-10^{-5}$, and consider the spatial domain $x \in[-L, L]$ with $L=4$ and the integration interval $t\in [0, T]$ with $T=20$.  The spatial and temporal domains are discretized using equal step sizes  with  $\Delta x=0.0808$ and  $\Delta t=0.009$.  The spatial grid points  are $x_i=-L+i \Delta x, i \in\{0,1,2, \ldots, N_1\}$ and the vector $\mathbf{u}$  is denoted as  $\mathbf{u}=[u^1, u^2, \ldots, u^{N_1}]$. 

\subsubsection{Semi-discrete system based on the first Hamiltonian form}
When we consider a semi-discrete system based on the first Hamiltonian form, the condition in Theorem \ref{Structure-preservation} is not satisfied. Consequently, the proposed methods do not exactly preserve the energy dissipation rate; however, the error should remain bounded and oscillatory for a diagonal matrix  $D(t)$ with small diagonal elements. The energy is discretized as: $$\mathcal{H}\left(\mathbf{u}\right)=\sum_{j=1}^{N_1} \big(\frac{\alpha}{3}(u^j)^3+\frac{\rho}{2}(u^j)^2-\frac{\nu}{2}\frac{u^{j+1}-u^{j}}{\Delta x}\big)\Delta x,$$ 
and 
$\partial_x$ is approximated by $D_1$. Thus, the semi-discrete system  is given by:
\begin{equation}\label{KdV-H1}
\dot{\mathbf{u}}= D_1\nabla H_1(\mathbf{u})-2 \gamma \mathbf{u},
\end{equation}
with $H_1(\mathbf{u})=\frac{\alpha}{3}\mathbf{u}^3+\frac{\rho}{2}\mathbf{u}^2+\nu D_2 \mathbf{u}$. 

Figure \ref{kdv-gamma1_Hamil1} presents the numerical results for the damped KdV equation \eqref{KdV-H1} using the proposed energy-preserving collocation methods of different orders. We observe that methods of different orders produce similar numerical solutions. Figure  \ref{solution-kdv_Hamil1} shows the evolution of numerical wave over time by the second-order method. Figure   \ref{energyDecay-kdv_Hamil1} illustrates the  exponential decay of both $H_1$ and  $H_2$ ($H_2(u)=\mathbf{u}^2/2$ is obtained by the discretization of Hamiltonian $\mathcal{H}_2$) over time.  Figure \ref{E_decay_rate_kdv_Hamil1} demonstrates that the error of the numerical energy dissipation rate remains bounded and oscillatory throughout the integration. Figure   \ref{order_kdv_Hamil1} confirms the order of the proposed methods.
\begin{figure}[H]
     \centering
     \begin{subfigure}[b]{0.49\textwidth}
         \centering
\includegraphics[width=\textwidth]{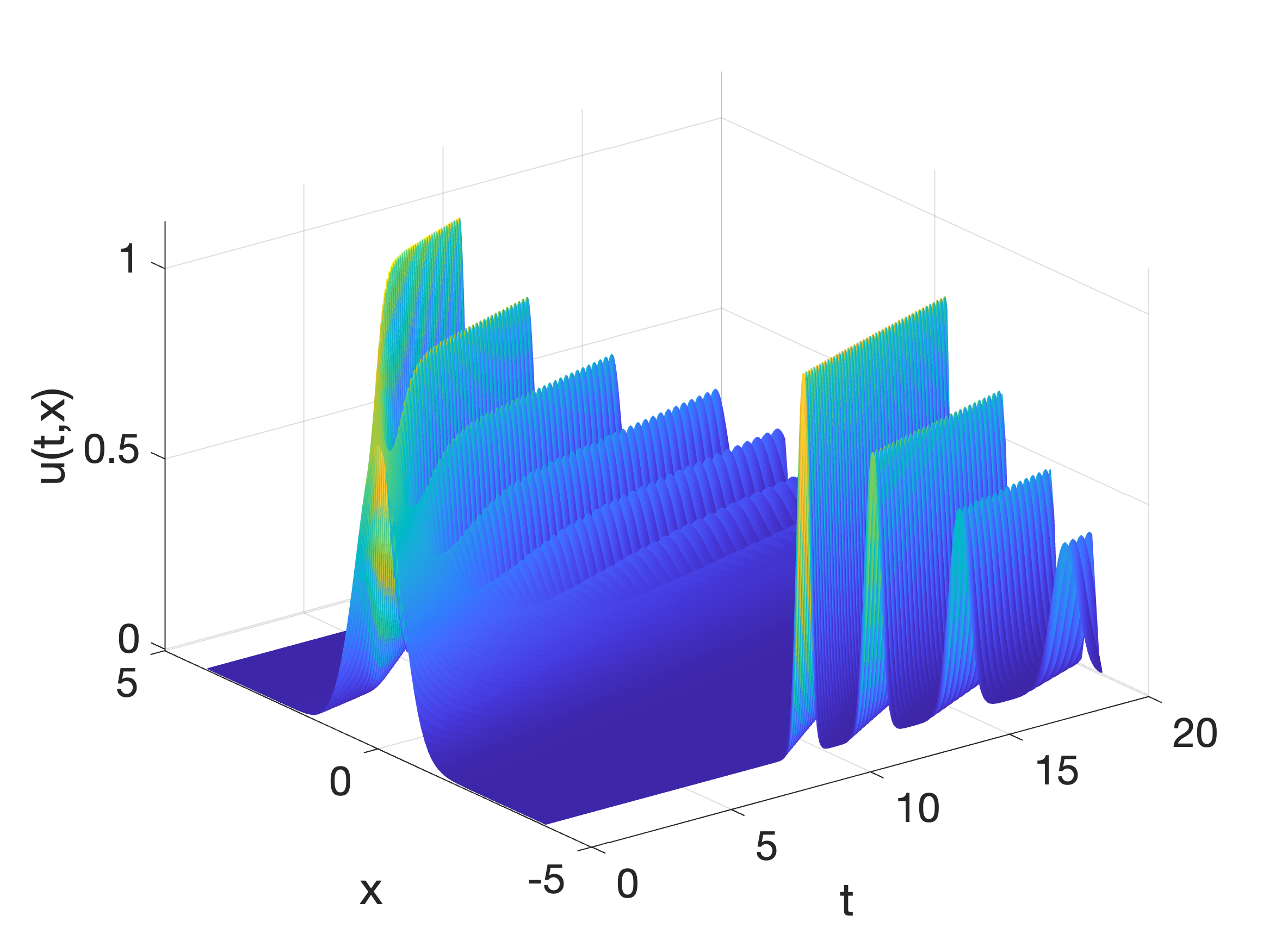}
         \caption{Solution}
         \label{solution-kdv_Hamil1}
     \end{subfigure}
     \hfill
       \begin{subfigure}[b]{0.49\textwidth}
         \centering
\includegraphics[width=\textwidth]{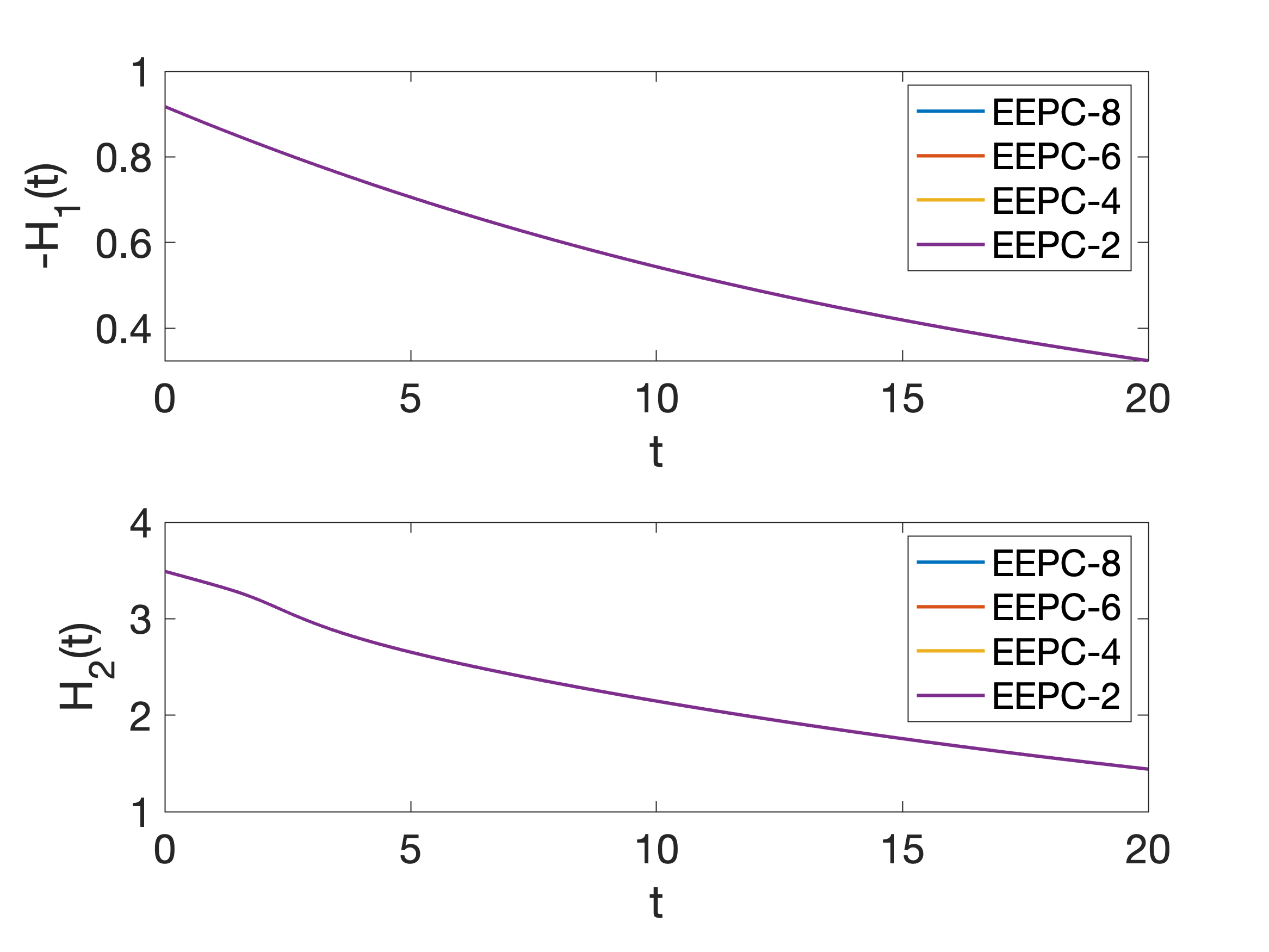}
         \caption{Decay of energy}
         \label{energyDecay-kdv_Hamil1}
     \end{subfigure}
       \hfill
          \begin{subfigure}[b]{0.49\textwidth}
         \centering
\includegraphics[width=\textwidth]{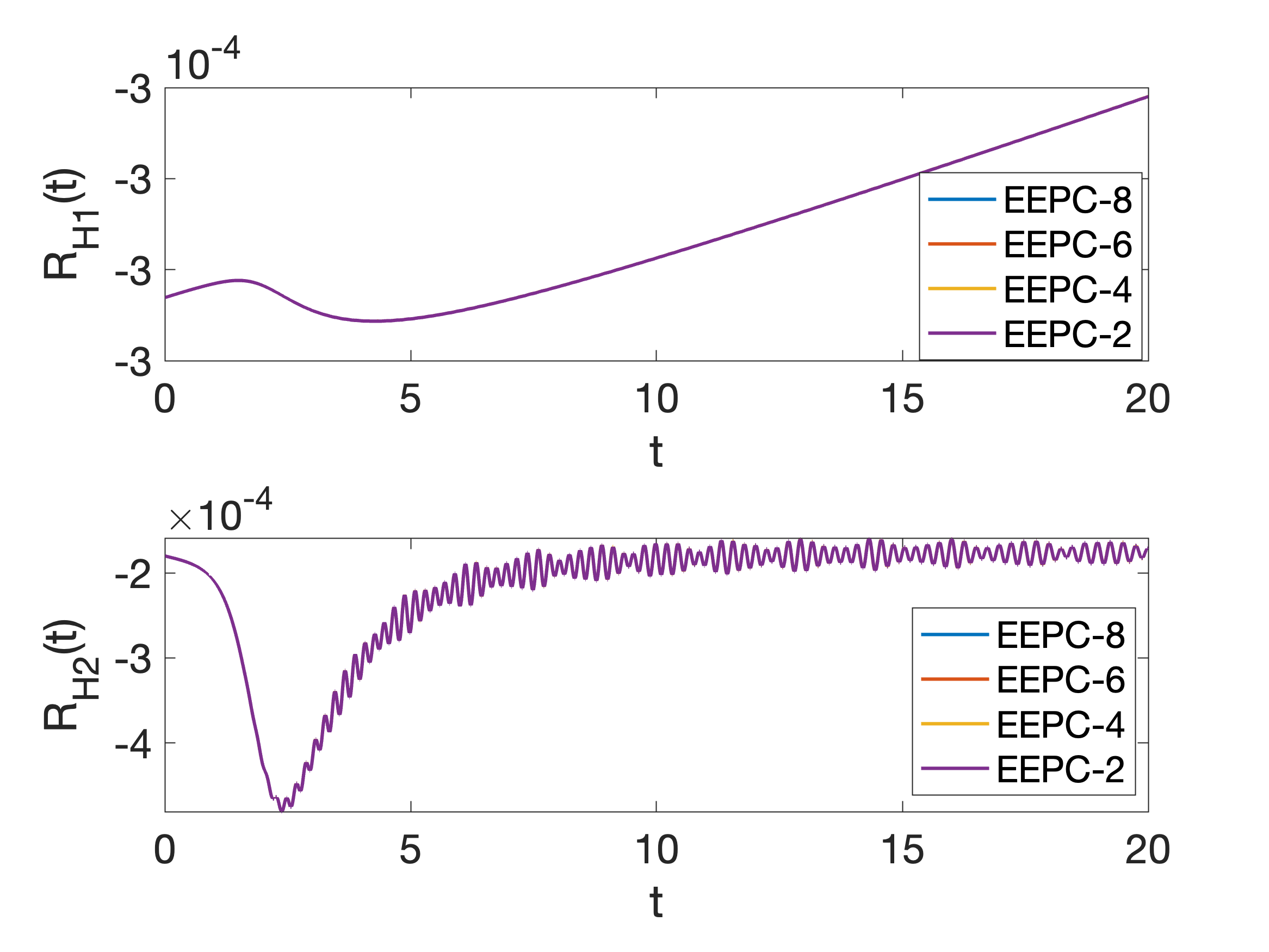}
         \caption{Conservation of energy decay rate}
         \label{E_decay_rate_kdv_Hamil1}
     \end{subfigure}
     \hfill
       \begin{subfigure}[b]{0.48\textwidth}
         \centering
\includegraphics[width=\textwidth]{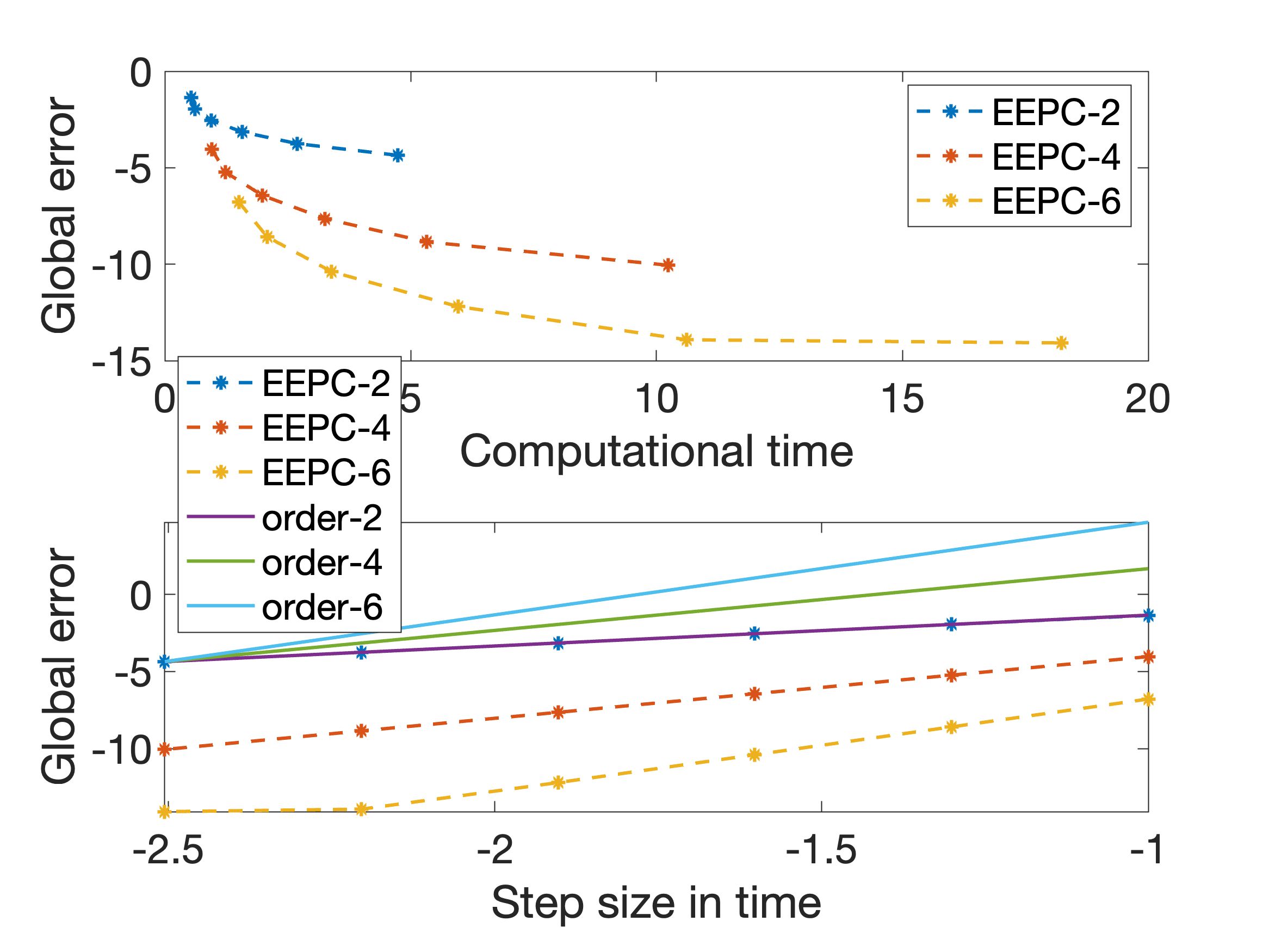}
         \caption{Order in time}
         \label{order_kdv_Hamil1}
     \end{subfigure}
     \hfill
        \caption{Plots of the exponential energy dissipation-preserving collocation methods of various orders  for the KdV equation \eqref{KdV-H1}, where $\gamma=0.01$ and $\Delta x=\pi/40$. In Fig  \ref{solution-kdv_Hamil1}, \ref{energyDecay-kdv_Hamil1} and  \ref{E_decay_rate_kdv_Hamil1}, fixed time step $\Delta t=0.009$ is used.}
        \label{kdv-gamma1_Hamil1}
\end{figure}

\subsubsection{Semi-discrete system based on the second Hamiltonian form}
 We consider the  semi-discrete Hamiltonian system based on the second Hamiltonian form. The term $u \partial_x+\partial_x u$ can be  discretized as
$$A(\mathbf{u})=\frac{1}{2 \Delta x}\left(\begin{array}{cccc}0 & u^1+u^2 & & -\left(u^1+u^{N_1}\right) \\ -\left(u^1+u^2\right) & \ddots & \ddots & \\ & \ddots & \ddots & u^{N_1-1}+u^{N_1} \\ u^1+u^{N_1} & & -\left(u^{N_1-1}+u^{N_1}\right) & 0\end{array}\right),$$ 
and the term $\partial_x^3$ can be discretized as $D_3=D_1*D_2$. Thus, the semi-discrete system is given by 
\begin{equation}\label{Eq-kdv-H2}
\dot{\mathbf{u}}= (\nu D_3+2\alpha/3A(\mathbf{u})+\rho D_1)\mathbf{u}-2 \gamma \mathbf{u},
\end{equation}
which can be written as in the form of \eqref{H-damp-ODE} with the skew-symmetric matrix as $S(\mathbf{u})=(\nu D_3+2\alpha/3A(\mathbf{u})+\rho D_1)$, the energy as $H_2(\mathbf{u})=\mathbf{u}^2/2$ and the coefficient matrix as $D(t)=2\gamma$. 

$H_2$ is a conformal invariant for the  damped Hamiltonian system \eqref{Eq-kdv-H2} 
since $H_2$ satisfies
$$
\frac{d}{d t} H_2(\mathbf{u}(t))=-4 \gamma H_2(\mathbf{u}(t)).
$$
Therefore, the solution evolves on a sphere shrinking at an exponential rate.  By Theorem \ref{Structure-preservation}, the proposed exponential methods can preserve the energy property, i.e., preserving the conformal invariant $H_2$. 

\textbf{Case I: $\textbf{D(t)}$ with constant equal diagonal elements $\gamma=0.01$.} We apply the proposed methods to  \eqref{Eq-kdv-H2} following Remark \ref{remark-S}.
Fig \ref{solution-kdv_Hamil2} illustrates the numerical solution evolving over time. Figure \ref{energyDecay-kdv_Hamil2} shows the dissipation of the energy function $H_2$ and $H_1$ over time. Here, we notice that the discrete energy $H_1$
does not strictly dissipate initially; however, it starts to dissipate stably after a few time steps. Figure \ref{E_decay_rate_kdv_Hamil2} demonstrates that the dissipation rate of energy $H_2$ can be kept to machine accuracy while the error in the dissipation rate of $H_1$ remains bounded and oscillates. Figure \ref{order_kdv_Hamil2} confirms the order of the proposed methods and also demonstrates the efficiency of the high-order method in achieving the same global error in less time.
\begin{figure}[H]
     \centering
     \begin{subfigure}[b]{0.49\textwidth}
         \centering
\includegraphics[width=\textwidth]{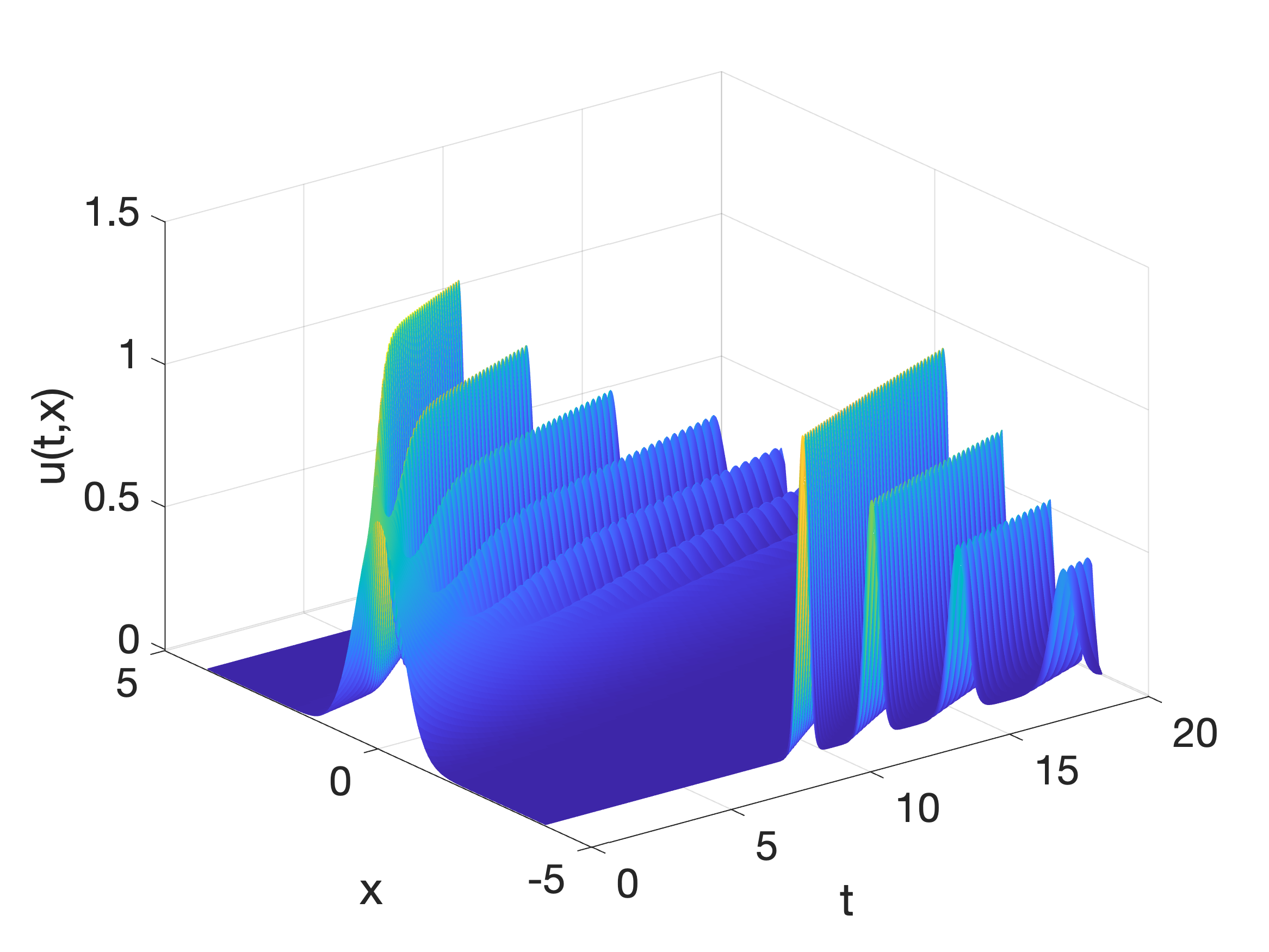}
         \caption{Solution}
         \label{solution-kdv_Hamil2}
     \end{subfigure}
     \hfill
       \begin{subfigure}[b]{0.49\textwidth}
         \centering
\includegraphics[width=\textwidth]{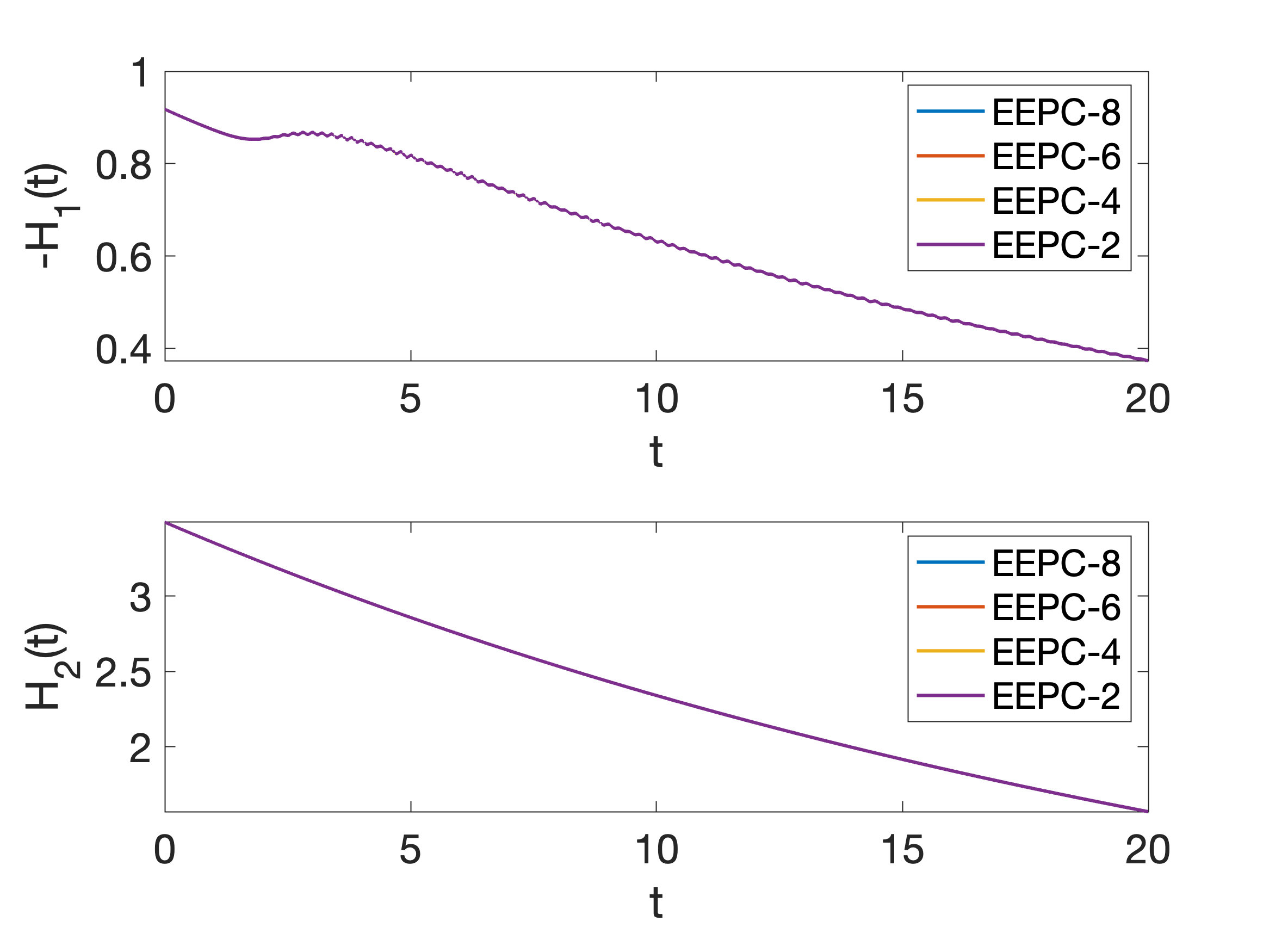}
         \caption{Decay of energy}
         \label{energyDecay-kdv_Hamil2}
     \end{subfigure}
       \hfill
          \begin{subfigure}[b]{0.49\textwidth}
         \centering
\includegraphics[width=\textwidth]{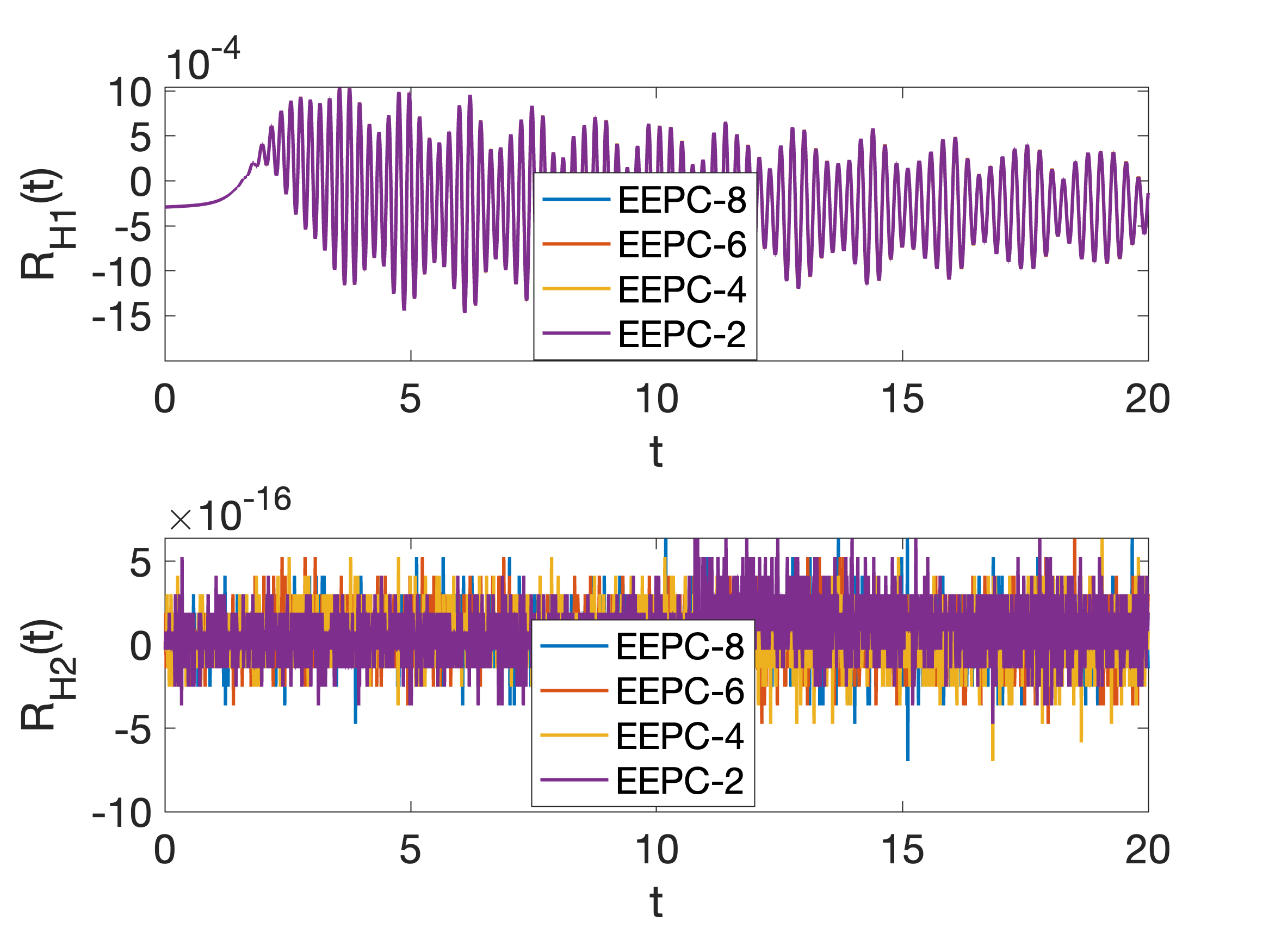}
         \caption{Conservation of energy decay rate}
         \label{E_decay_rate_kdv_Hamil2}
     \end{subfigure}
     \hfill
       \begin{subfigure}[b]{0.48\textwidth}
         \centering
\includegraphics[width=\textwidth]{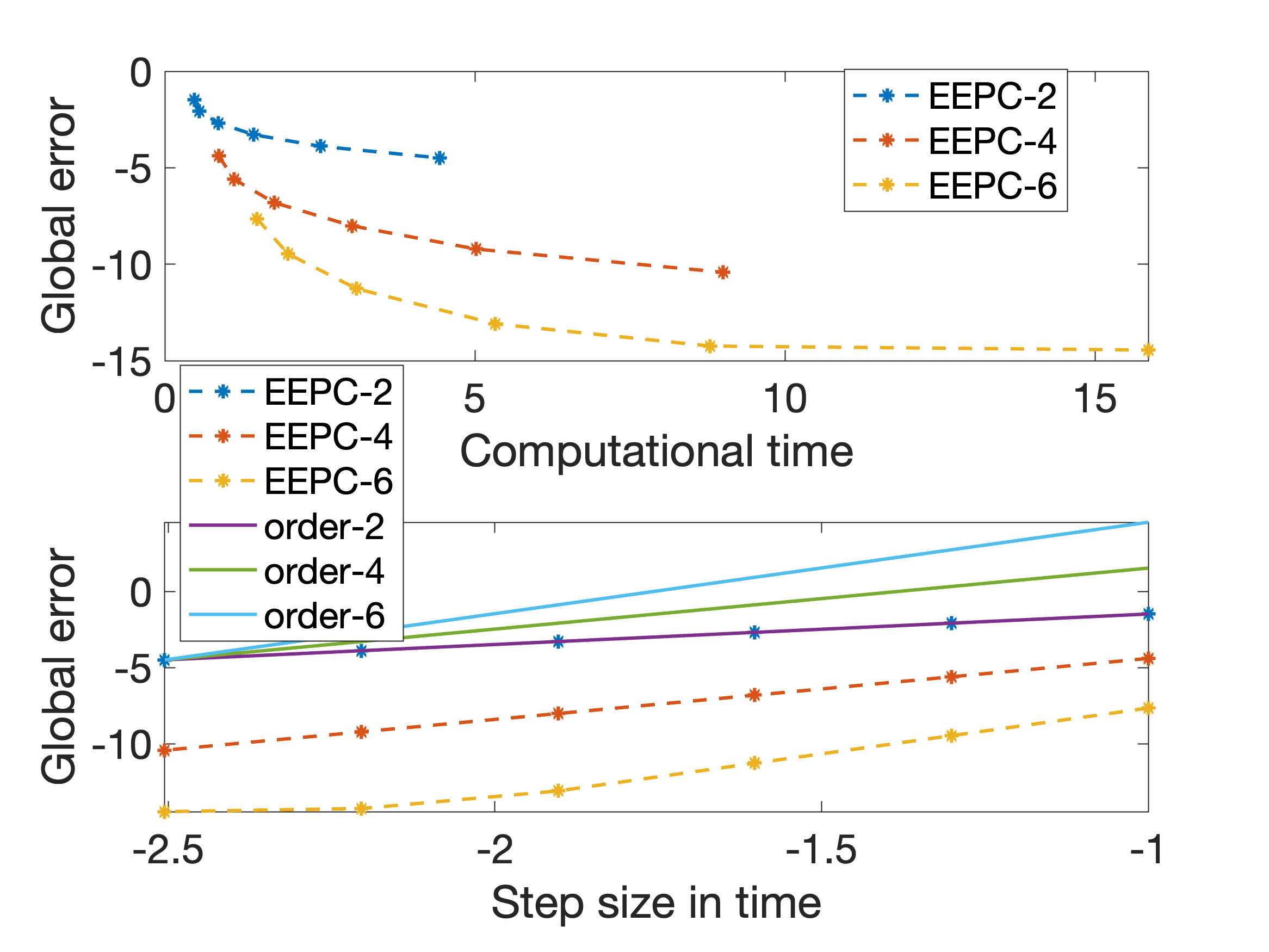}
         \caption{Order in time}
         \label{order_kdv_Hamil2}
     \end{subfigure}
     \hfill
        \caption{Plots of the exponential energy-preserving collocation methods of various orders  for the KdV equation with the second Hamiltonian form \eqref{Eq-kdv-H2}, where $\gamma=0.01$ and $\Delta x=\pi/40$. In Fig  \ref{solution-kdv_Hamil2}, \ref{energyDecay-kdv_Hamil2} and  \ref{E_decay_rate_kdv_Hamil2}, fixed time step $\Delta t=0.009$ is used.}
        \label{kdv-gamma1_Hamil2}
\end{figure}
With the second Hamiltonian form, the energy is quadratic, satisfying the conditions of  Theorem \ref{Structure-preservation} for all diagonal matrix $D(t)$ with equal diagonal elements. However, this condition is not maintained for diagonal matrices $D(t)$ with unequal diagonal elements, even in the simplest cases where $D(t)$ are constant matrices. To numerically  demonstrate these scenarios, we consider the second Hamiltonian form with  two other sets of $D(t)$ as follows.

\textbf{Case II: $\textbf{D(t)}$ with constant unequal diagonal elements $\gamma$.} 
We introduce a random perturbation to each default diagonal element of $D(t)$  in Case I  within $10\%$ of the default value, resulting in unequal diagonal elements. As seen in Fig \ref{solution-kdv_Hamil2_gamma2}, there is a consistent damping of the numerical wave amplitude over time. Figure \ref{energyDecay-kdv_Hamil2_gamma2} depicts the decay of energy, which appears to be exponential. Figure \ref{E_decay_rate_kdv_Hamil2_gamma2} illustrates that the conservation property described in  Theorem \ref{Structure-preservation} is not maintained once the condition of the energy (or invariant in a more general sense) is not satisfied. However, we observe a constant difference between machine accuracy and the error computed from formula \eqref{dissp-rate-average} based on  the numerical dissipation rate of $H_2$.  This observation inspires further exploration of the correct underlying conservation law using more advanced technique such as deep learning \cite{eidnes2023pseudo}. Figure \ref{order_kdv_Hamil2_gamma2} confirms the order of the proposed methods and demonstrates that higher-order methods are more effective.

\begin{figure}[H]
     \centering
     \begin{subfigure}[b]{0.49\textwidth}
         \centering
\includegraphics[width=\textwidth]{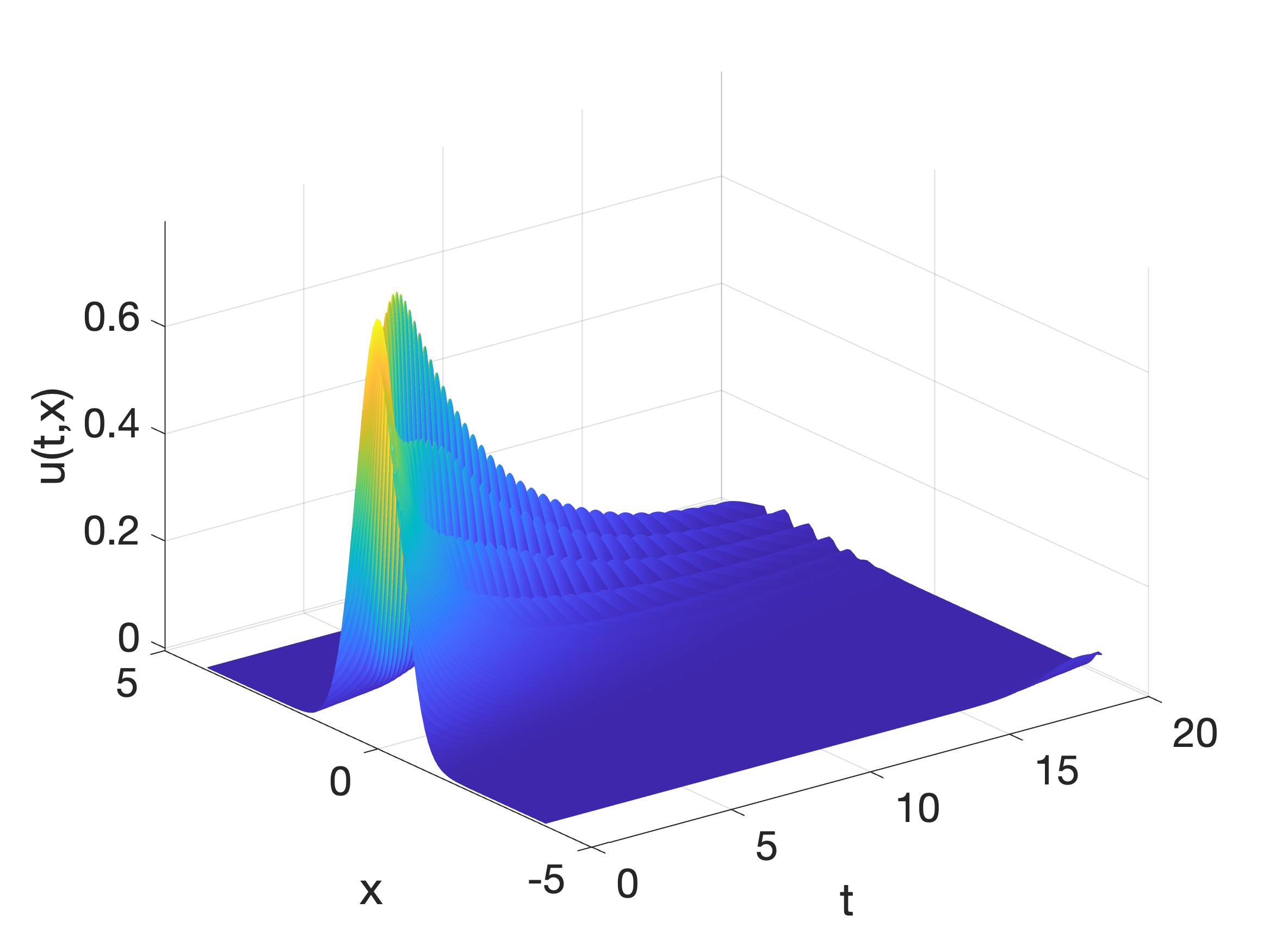}
         \caption{Solution}
         \label{solution-kdv_Hamil2_gamma2}
     \end{subfigure}
     \hfill
       \begin{subfigure}[b]{0.49\textwidth}
         \centering
\includegraphics[width=\textwidth]{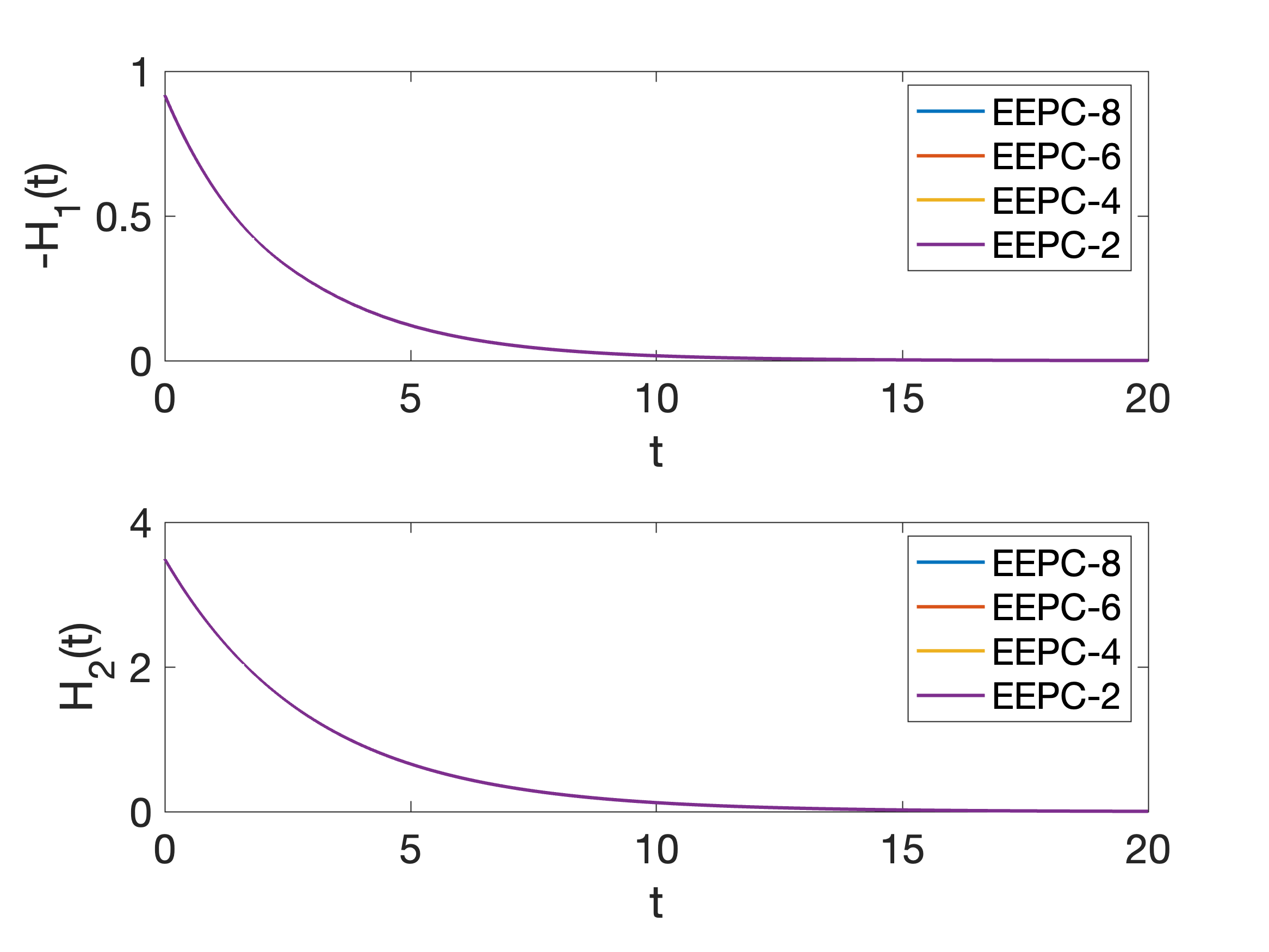}
         \caption{Decay of energy}
         \label{energyDecay-kdv_Hamil2_gamma2}
     \end{subfigure}
       \hfill
          \begin{subfigure}[b]{0.49\textwidth}
         \centering
\includegraphics[width=\textwidth]{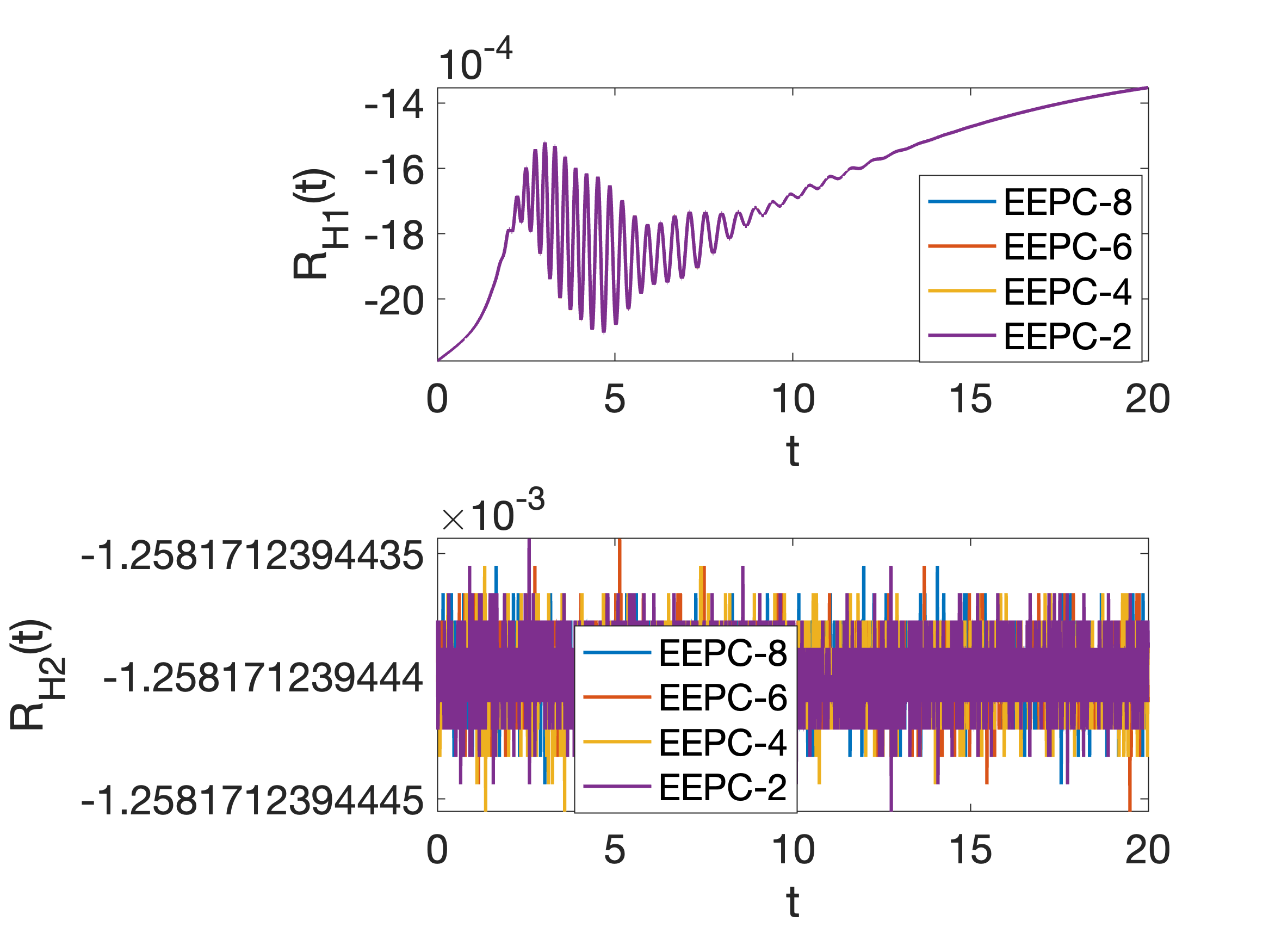}
         \caption{Conservation of energy decay rate}
         \label{E_decay_rate_kdv_Hamil2_gamma2}
     \end{subfigure}
     \hfill
       \begin{subfigure}[b]{0.48\textwidth}
         \centering
\includegraphics[width=\textwidth]{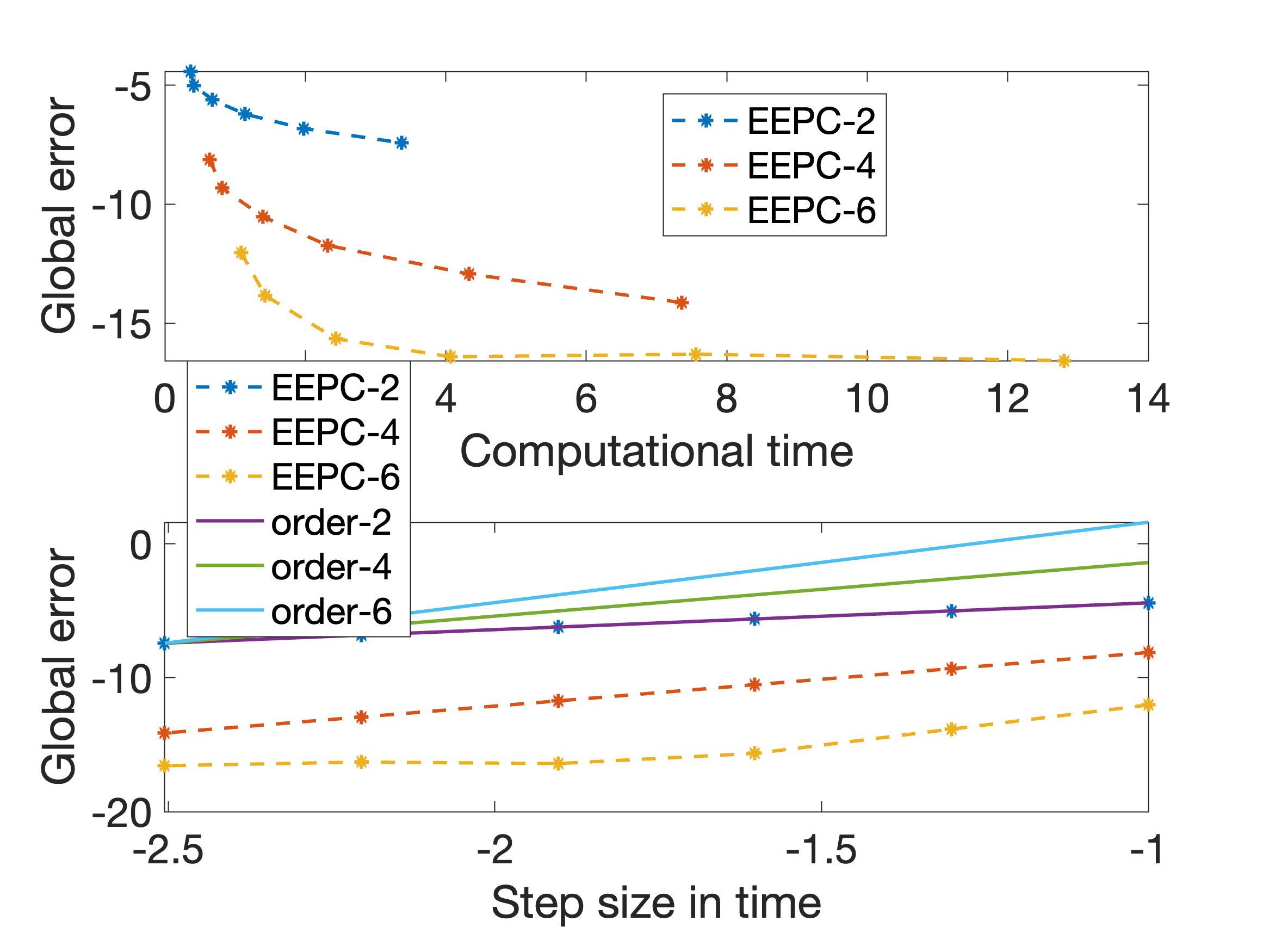}
         \caption{Order in time}
         \label{order_kdv_Hamil2_gamma2}
     \end{subfigure}
     \hfill
        \caption{Plots of the exponential energy dissipation-preserving collocation methods of various orders  for the KdV equation with the second Hamiltonian form, where the diagonal elements of $D(t)$ are random numbers around $0.2$ and $\Delta x=\pi/40$. In Fig  \ref{solution-kdv_Hamil2_gamma2}, \ref{energyDecay-kdv_Hamil2_gamma2} and  \ref{E_decay_rate_kdv_Hamil2_gamma2}, fixed time step $\Delta t=0.009$ is used.}
        \label{kdv-gamma2_Hamil2}
\end{figure}

\textbf{Case III: $\textbf{D(t)}$ with time-dependent equal diagonal elements $\gamma$.} 
We consider equal diagonal elements $\gamma_k(t)=e^{-t}/2$.
Figure \ref{kdv-gamma3_Hamil2} shows the numerical solution, where we observe a decrease in wave amplitude followed by the spreading of the initial peak into smaller secondary peaks over time.   In Figure \ref{energyDecay-kdv_Hamil2_gamma3}, the energy decay appears exponential. Figure \ref{kdv-gamma3_Hamil2} demonstrates that the conservation property in Theorem \ref{Structure-preservation} is maintained for diagonal matrices $D(t)$ as long as their diagonal elements are equal, regardless of whether they are constant. Figure \ref{order_kdv_Hamil2_gamma3} confirms the order of the proposed methods, also showing that higher-order methods are more effective.

\begin{figure}[H]
     \centering
     \begin{subfigure}[b]{0.49\textwidth}
         \centering
\includegraphics[width=\textwidth]{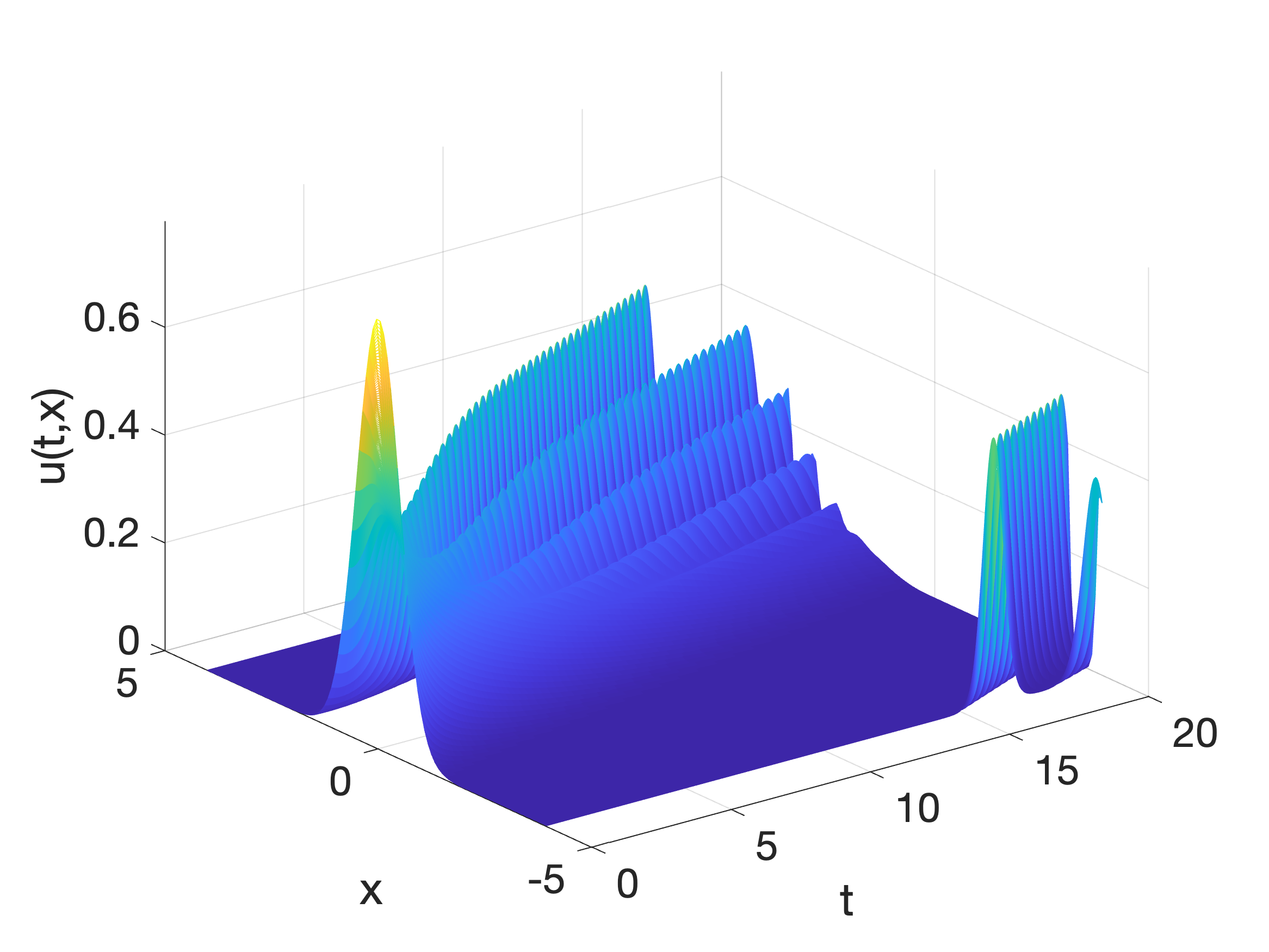}
         \caption{Solution}
         \label{solution-kdv_Hamil2_gamma3}
     \end{subfigure}
     \hfill
       \begin{subfigure}[b]{0.49\textwidth}
         \centering
\includegraphics[width=\textwidth]{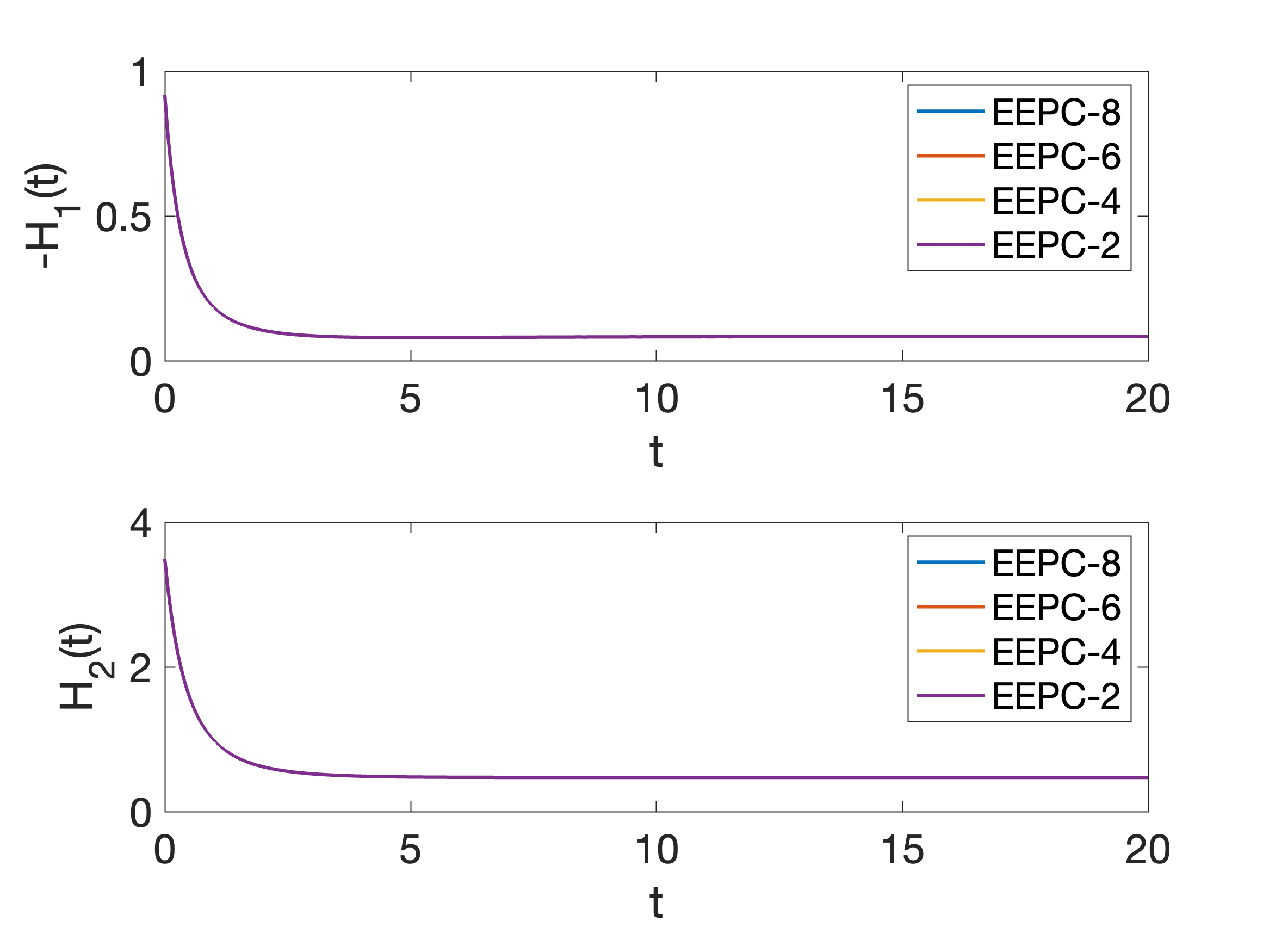}
         \caption{Decay of energy}
         \label{energyDecay-kdv_Hamil2_gamma3}
     \end{subfigure}
       \hfill
          \begin{subfigure}[b]{0.49\textwidth}
         \centering
\includegraphics[width=\textwidth]{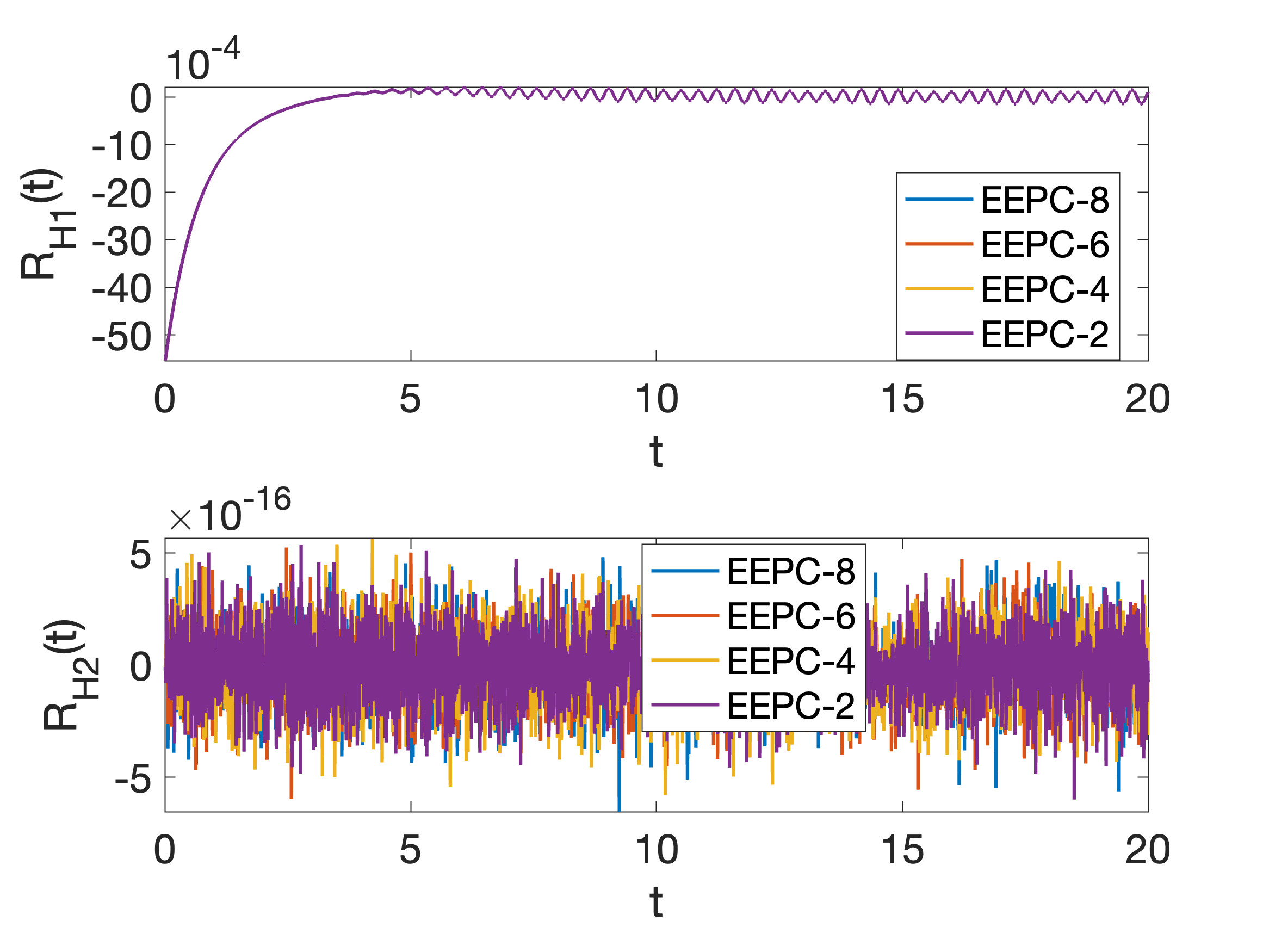}
         \caption{Conservation of energy decay rate}
         \label{E_decay_rate_kdv_Hamil2_gamma3}
     \end{subfigure}
     \hfill
       \begin{subfigure}[b]{0.48\textwidth}
         \centering
\includegraphics[width=\textwidth]{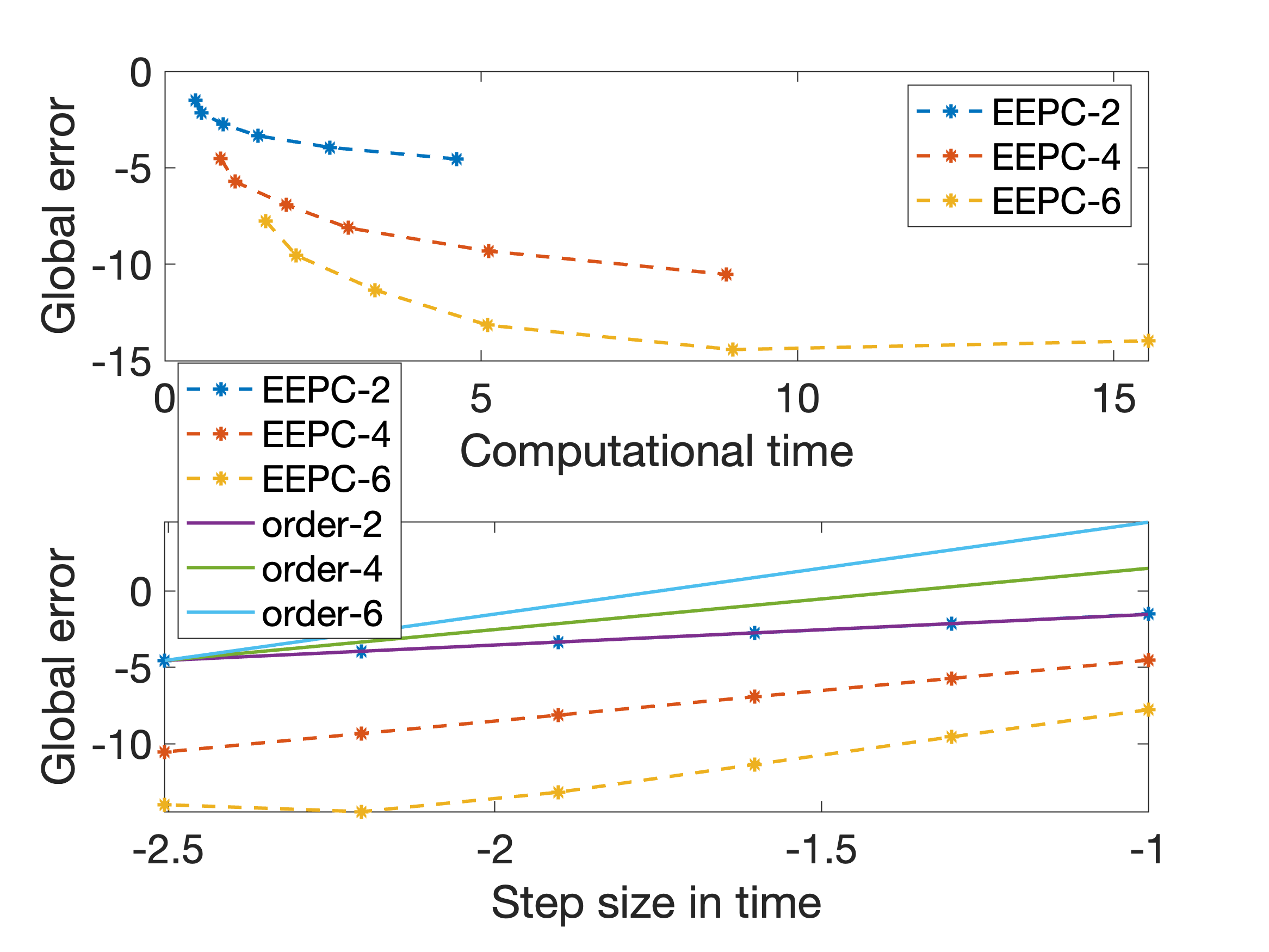}
         \caption{Order in time}
         \label{order_kdv_Hamil2_gamma3}
     \end{subfigure}
     \hfill
        \caption{Plots of the energy dissipation-preserving collocation methods of various orders  for the KdV equation with the second Hamiltonian form, where all diagonal elements of matrix $D(t)$ are equal to $e^{-t}$ and $\Delta x=\pi/40$. In Fig  \ref{solution-kdv_Hamil2_gamma3}, \ref{energyDecay-kdv_Hamil2_gamma3} and  \ref{E_decay_rate_kdv_Hamil2_gamma3}, fixed time step $\Delta t=0.009$ is used.}
        \label{kdv-gamma3_Hamil2}
\end{figure}




\section{Conclusion}\label{conclusion}
In this study, arbitrarily high-order structure-preserving methods have been developed for damped Hamiltonian systems by combining exponential integrators with energy-preserving collocation methods. These integrators maintain the energy dissipation ratio introduced by damping terms, exhibiting notable advantages in terms of efficiency, accuracy, and stability. This makes them well-suited for solving complex and large-scale problems. The presented results are based on damped Hamiltonian systems with a diagonal coefficient matrix $D(t)$ of special forms. For certain cases, such as the KdV equation with 
$D(t)$ having constant unequal diagonal elements, numerical simulations suggest the existence of particular energy dissipation laws, despite the challenge in uncovering such energy conservation laws for the true solution. This observation prompts further investigation into the underlying physical laws using more advanced techniques. Additionally, extending suitable structure-preserving methods for more general Hamiltonian systems with varying damping terms remains an open research question for the future.

\bibliographystyle{elsarticle-num}
\bibliography{main}

\end{document}